\documentclass[format=acmsmall,screen=true,review=false]{acmart} 
\usepackage[utf8]{inputenc}
\usepackage{fontenc}
\usepackage[english]{babel}

\acmJournal{TOMS}
\setcopyright{rightsretained}

\numberwithin{equation}{section}
\usepackage{amsthm}
\usepackage{hyperref}
\usepackage[capitalise]{cleveref}
\usepackage{xcolor}
\usepackage[shortlabels]{enumitem}
\setlist[enumerate]{
itemindent=0pt,
leftmargin=25pt
}
 
\usepackage[section]{algorithm}
\usepackage{algpseudocode,algorithmicx}
\usepackage{tikz-cd}

\DeclareMathAlphabet{\mathpzc}{OT1}{pzc}{m}{it} 

\newcommand{\der}{\mathrm{d}}
\newcommand{\Tang}{\mathrm{T}}
\newcommand{\rank}{\mathrm{rank}}
\newcommand{\diag}{\mathrm{diag}}
\newcommand{\Id}{\mathrm{Id}}
\newcommand{\FF}{\mathbb{F}}
\newcommand{\herm}{{\mathsf{H}}}
\newcommand{\trans}{{\mathsf{T}}}

\newcommand{\vect}[1]{\mathbf{#1}}
\newcommand{\tensor}[1]{\mathpzc{#1}}
\newcommand{\VS}[1]{\mathrm{#1}}
\newcommand{\Var}[1]{\mathcal{#1}}
\newcommand{\BigO}[1]{\mathcal{O}\left(#1\right)}

\newcommand{\sign}{\mathrm{sign}}
\newcommand{\changed}[1]{{#1}}

\crefname{equation}{}{}
\Crefname{equation}{Equation}{Equations} 

\title[An algorithm for Grassmann decomposition]{A chiseling algorithm for low-rank Grassmann decomposition of skew-symmetric tensors}

\author{Nick Vannieuwenhoven}
\orcid{0000-0001-5692-4163}
\affiliation{%
\institution{KU Leuven}
\department{Department of Computer Science} 
\streetaddress{Celestijnenlaan 200A}
\city{Heverlee}
\postcode{3001}
\country{Belgium}
}
\email{nick.vannieuwenhoven@kuleuven.be}

\setcitestyle{numbers,sort&compress}

\allowdisplaybreaks
\begin{document}

\begin{abstract}
A numerical algorithm to decompose \changed{an exact} low-rank skew-symmetric tensor into a sum of elementary skew-symmetric tensors is introduced. The algorithm uncovers this Grassmann decomposition based on linear relations that are encoded by the kernel of the differential of the natural action of the general linear group on the tensor, following the ideas of [Brooksbank, Kassabov, and Wilson, \textit{Detecting \changed{null} patterns in tensor data}, arXiv:2408.17425\changed{v2, 2025}]. The Grassmann decomposition can be recovered, up to scale, from a diagonalization of a generic element in this kernel. Numerical experiments illustrate that the algorithm is computationally efficient and quite accurate for mathematically low-rank tensors.
\end{abstract} 



\maketitle

\section{Introduction}\label{sec_introduction}

\emph{Alternating}, \emph{antisymmetric}, or \emph{skew-symmetric tensors} \changed{are elements of} the \emph{alternating}, \emph{exterior}, or \emph{wedge product} of an $n$-dimensional vector space $\VS{V}$ over a field $\FF$ \citep{Lang2002,Greub1978,Landsberg2012,Winitzki2023}.
This wedge product can be defined for vectors $\vect{v}_1, \ldots, \vect{v}_d \in \VS{V}$ as
\begin{equation}\label{eqn_wedge_product}
 \vect{v}_1 \wedge\cdots\wedge \vect{v}_d := \changed{\frac{1}{d!}} \sum_{\sigma\in\mathfrak{S}([d])} \operatorname{sign}(\sigma)\, \vect{v}_{\sigma_1} \otimes\cdots\otimes \vect{v}_{\sigma_d},
\end{equation}
where $\sign(\sigma)$ is the sign of the permutation $\sigma$ from the set of permutations $\mathfrak{S}$ on $[d]:=\{1,\dots,d\}$, $\sigma_i := \sigma(i)$, and $\otimes$ is the tensor product \citep{Lang2002,Greub1978,Landsberg2012,Winitzki2023}. 
A skew-symmetric tensor of the form \cref{eqn_wedge_product} \changed{is called} an \emph{elementary skew-symmetric tensor}.

\changed{A first} viewpoint on (elementary) skew-symmetric tensors is that they are a space of tensors satisfying certain linear symmetries. Specifically, a transposition of the two vectors at positions $i\ne j$ only swaps the sign of an elementary tensor, i.e.,
\begin{equation} \label{eqn_alternating}
 \vect{v}_1 \wedge\dots \wedge \vect{v}_i\wedge\dots\wedge \vect{v}_j\wedge \dots\wedge\vect{v}_d = 
 -\vect{v}_1 \wedge\dots \wedge \vect{v}_j\wedge\dots\wedge \vect{v}_i\wedge \dots\wedge\vect{v}_d,
\end{equation}
for any $i\ne j$, which explains the terminology ``alternating.''
The space $\wedge^2 \VS{V}$ can be interpreted as the familiar space of skew-symmetric matrices.

\changed{A second} viewpoint on elementary skew-symmetric tensors comes from their relation to $d$-dimensional linear subspaces of $\VS{V}$. It is an elementary fact that if $\VS{U} \subset \VS{V}$ is a subspace with a basis $(\vect{u}_1,\dots,\vect{u}_d)$, then $(\vect{u}_1',\dots,\vect{u}_d')$ is a basis of $\VS{U}$ if and only if there exists a nonzero scalar $\alpha$ such that
\[
 \vect{u}_1 \wedge\dots\wedge \vect{u}_d = \alpha \cdot \vect{u}_1'\wedge\dots\wedge \vect{u}_d';
\]
see \cref{cor_det_transform}.
\changed{In other words, there is a bijective correspondence between the linear subspace spanned by $d$ linearly independent vectors $\vect{u}_1,\dots,\vect{u}_d$ and the affine cone over the elementary skew-symmetric tensor $\vect{u}_1 \wedge\dots\wedge \vect{u}_d$, namely
\[
\operatorname{span}(\vect{u}_1, \dots, \vect{u}_d) \simeq [\vect{u}_1 \wedge\dots\wedge \vect{u}_d] := \{ \alpha \vect{u}_1 \wedge\dots\wedge \vect{u}_d \mid \alpha \in \FF\setminus\{0\} \}.
\]
This identification of $d$-dimensional subspaces with the affine cone over elementary skew-symmetric tensors} is called the \emph{Pl\"ucker embedding} of the \emph{Grassmannian manifold} $\mathrm{Gr}(d,\VS{V})$ of $d$-dimensional linear subspaces of an $n$-dimensional $\VS{V}$ into the space of projective skew-symmetric tensors. \changed{This viewpoint highlights that} the space of all nonzero elementary skew-symmetric tensors
\[
 \Var{G}_n^d := \{ \vect{v}_1 \wedge\dots\wedge \vect{v}_d \mid \vect{v}_i \in \VS{V},\, i=1,\dots,d\}\setminus\{0\} \subset \VS{V}\otimes\dots\otimes\VS{V}
\]
is a smooth manifold. Since $\Var{G}^d_n$ is the cone over the Grassmannian $\mathrm{Gr}(d,\VS{V})$ (in its Pl\"ucker embedding), its dimension is $\dim \Var{G}_n^d = 1+d(n-d)$ \cite[p.~138]{Harris1992}.
The linear span of $\Var{G}_n^d$ is the $\binom{n}{d}$-dimensional vector space $\wedge^d \VS{V}$ of order-$d$ \emph{skew-symmetric tensors}. \changed{It is a vector subspace of the $n^d$-dimensional vector space $\VS{V}\otimes\dots\otimes\VS{V}$.
See Harris \cite[Lecture 6]{Harris1992} for further information on this viewpoint.}

\changed{A third viewpoint on elementary skew-symmetric tensors emerges when we view them as \emph{alternating multilinear maps} \cite[Chapter 5]{Greub1978}. Consider an elementary skew-symmetric tensor $f_1 \wedge\dots\wedge f_d \in \wedge^d \VS{V}^*$, where $\VS{V}^*$ is the dual vector space of $\VS{V}$. By definition of the alternating tensor product \cite{Greub1978}, this tensor uniquely corresponds to the alternating multilinear map
\[
f_1 \wedge\dots\wedge f_d : \VS{V} \times\dots\times \VS{V} \to \FF, \quad 
(\vect{v}_1,\dots,\vect{v}_d) \mapsto \frac{1}{d!} \sum_{\sigma\in\mathfrak{S}([d])} \operatorname{sign}(\sigma) f_{\sigma_1}(\vect{v}_1) \cdots f_{\sigma_d}(\vect{v}_d).
\]
Recalling the definition of the determinant of a $d \times d$ matrix $A$, namely 
\[
\det(A) = \sum_{\sigma\in\mathfrak{S}([d])} \operatorname{sign}(\sigma) a_{\sigma_1,1} \cdots a_{\sigma_d,d},
\]
we see that 
\begin{align}\label{eqn_alt_map_det}
(f_1 \wedge\dots\wedge f_d)(\vect{v}_1,\dots,\vect{v}_d) = \frac{1}{d!} \det 
  \begin{bmatrix} 
    f_1(\vect{v}_1) & \dots & f_1(\vect{v}_d) \\ 
    \vdots & & \vdots \\  
    f_d(\vect{v}_1) & \dots & f_d(\vect{v}_d)
  \end{bmatrix}.
\end{align}
Intuitively, this means that the elementary skew-symmetric tensor $f_1 \wedge\dots\wedge f_d$ essentially represents a generalized determinant function. Indeed, the determinant of $d \times d$ matrices is the special case
\[
\det(A) = d! \cdot (\vect{e}_1^\trans{} \wedge\dots\wedge \vect{e}_d^\trans{})(\vect{a}_1,\dots,\vect{a}_d),
\]
where $\vect{a}_i$ denotes the $i$th column of $A \in \FF^{d\times d}$ and $(\vect{e}_1,\dots,\vect{e}_d)$ is the canonical basis of the space of column vectors $\FF^n$; the row vectors $\vect{e}_i^\trans{}$ define linear functions through matrix multiplication.}

The topic of this article is the \emph{Grassmann decomposition} of a skew-symmetric tensor $\tensor{A} \in \wedge^d \VS{V}$ into a sum of elementary skew-symmetric tensors. That is, given
\begin{equation}\label{eqn_grassmann_decomp}
\tensor{A} = \tensor{A}_1 + \dots + \tensor{A}_r,
\end{equation}
can we compute the set of elementary skew-symmetric tensors $\{\tensor{A}_1, \dots, \tensor{A}_r\}\subset \Var{G}^d_n$?
The smallest $r \in \mathbb{N}$ for which this is possible is called the \emph{Grassmann rank} of $\tensor{A}$. For brevity, it will be referred to as the \emph{Gr-rank} of $\tensor{A}$. 

\subsection*{Connection to applications}
Grassmann decompositions of skew-symmetric tensors have an interesting connection to quantum physics; the following discussion is based on \citep{MVHC2023,ESBL2002}.
A single physical particle (boson or fermion) $v$ is postulated to admit a \emph{state} that can be modeled mathematically by a vector $\vect{v}$ in a vector space $\VS{V}$ whose dimension depends on the information carried by the particle $v$, such as its magnetic, principal, and spin quantum numbers. The joint state of a quantum system comprised of $d$ \emph{indistinguishable} fermions is mathematically modeled by a skew-symmetric tensor in $\wedge^d \VS{V}$.
In particular, a quantum system of $d$ \emph{nonentangled} pure fermions $v_i$ with respective states $\vect{v}_i \in \VS{V}$ is represented mathematically as the elementary skew-symmetric tensor $\vect{v}_1 \wedge\dots\wedge \vect{v}_d$.
Fermions obey the \emph{Pauli exclusion principle}, which postulates that fermions cannot possess the same physical state. Mathematically this is captured by the fact that
\[
 \vect{v}_1 \wedge \vect{v}_1 \wedge \vect{v}_3 \wedge\dots\wedge\vect{v}_d = 0,
\]
which is a consequence of \cref{eqn_alternating}.
A fermionic quantum system can be considered \emph{nonentangled} if and only if its state is represented by an elementary skew-symmetric tensor, i.e., if its Gr-rank is $1$ \citep{ESBL2002}.\footnote{In the physics literature, the rank is called the \emph{Slater rank} and the Grassmann decomposition, a \emph{Slater decomposition}.} Since the space $\wedge^d \VS{V}$ is stratified by Gr-rank, the latter provides one natural mathematical measure of the entanglement of a fermionic quantum system, \changed{though other measures of entangledness have been proposed as well} \citep{ESBL2002}.

\changed{Another application of Grassmann decompositions are efficient algorithms for evaluating arbitrary alternating multilinear maps $f : \VS{V}\times\dots\times\VS{V} \to \FF$. Recall that the vector space of all alternating multilinear maps is isomorphic to the vector space of alternating tensors $\wedge^d \VS{V}^*$ \cite{Greub1978}. Therefore, $f$ can be identified with a skew-symmetric tensor
\begin{align}\label{eqn_general_alt_map}
\tensor{F} = \sum_{1 \le i_1 < i_2 < \dots < i_d \le n} f_{i_1 \dots i_d} \vect{e}_{i_1}^\trans{} \wedge \vect{e}_{i_2}^\trans{} \wedge \dots \wedge \vect{e}_{i_d}^\trans{},
\end{align}
where $\{ \vect{e}_{i_1}^\trans{}\wedge\vect{e}_{i_2}^\trans{}\wedge\dots\wedge\vect{e}_{i_d}^\trans{} \mid 1 \le i_1 < i_2 < \dots < i_d \le n \}$ is the tensor product basis of $\wedge^d \VS{V}$ induced by the basis $(\vect{e}_1^\trans{},\dots,\vect{e}_n^\trans{})$ of $\VS{V}^*$. The universal property of the alternating product \cite{Greub1978} ensures that $\tensor{F}(\vect{v}_1\otimes\dots\otimes\vect{v}_d)=f(\vect{v}_1,\dots,\vect{v}_d)$ for all $\vect{v}_i \in \VS{V}$.
Crucially, evaluating $f$ as suggested by \cref{eqn_general_alt_map} requires the evaluation of the $\binom{n}{d}$ elementary alternating multilinear maps $\vect{e}_{i_1}^\trans{} \wedge\dots\wedge \vect{e}_{i_d}^\trans{}$, after which the resulting scalars are all summed. We have seen in \cref{eqn_alt_map_det} that evaluating an elementary skew-symmetric multilinear map reduces to the determinant of a $d \times d$ matrix, obtained as the product of $d \times n$ and $n \times d$ matrices. We conclude that an arbitrary alternating multilinear map can be evaluated asymptotically with no more than $(d n^2 + d^3 + 1) \binom{n}{d}$ elementary operations, assuming Gauss elimination is used to evaluate the matrix determinant. By contrast, if a Gr-rank $r$ Grassmann decomposition of $\tensor{F}$ is known, i.e., 
\(
\tensor{F} = \sum_{i=1}^r g^1_i \wedge g^2_i \wedge\dots\wedge g^d_i,
\)
where $g_i^k : \VS{V} \to \FF$ are linear forms, then the complexity reduces to the evaluation of only $r$ elementary skew-symmetric multilinear maps. That is, the asymptotic complexity decreases to $\mathcal{O}( (d n^2 + d^3 + 1)r )$ operations.}

\changed{\subsection*{Related decompositions}
The Grassmann decomposition is a specific instance of a broad class of \emph{rank decompositions} \cite{Landsberg2012,BCCGO2018} of tensors. A rank decomposition of a tensor $\tensor{A} \in \VS{V}_1 \otimes\dots\otimes \VS{V}_d$, where $\VS{V}_i$ are vector spaces, with respect to a variety or manifold $\Var{M} \subset \VS{V}_1\otimes\dots\otimes \VS{V}_d$ is an expression of the form
\[
\tensor{A} = \tensor{A}_1 + \dots + \tensor{A}_r,\quad\text{with } \tensor{A}_1,\dots,\tensor{A}_r \in \Var{M},
\]
which is of minimal length among all such expressions. This minimal length is called the $\Var{M}$-rank of $\tensor{A}$. 
By the foregoing definitions, we see that the Grassmann decomposition is a rank decomposition with respect to the Grassmannian $\Var{G}_n^d \subset \wedge^d \VS{V} \subset \VS{V}\otimes\dots\otimes\VS{V}$.

Several theoretical properties of rank decompositions have been studied in great generality in applied algebraic geometry \cite{Landsberg2012,BCCGO2018}. In particular, the question of the \emph{identifiability} of a rank decomposition has been mostly resolved for a large class of varieties by recent breakthroughs \cite{MM2024,Ballico2024,TBC2023}. Recall that a tensor $\tensor{A}$ is $r$-identifiable with respect to $\Var{M}$ if there is a unique set $\{\tensor{A}_1,\dots,\tensor{A}_r\} \subset \Var{M}$ of cardinality $r$ such that $\tensor{A} = \tensor{A}_1 +\dots+ \tensor{A}_r$. The first-order \emph{sensitivity} of a rank decomposition $\{\tensor{A}_1,\dots,\tensor{A}_r\}$ relative to rank-preserving perturbations of the tensor $\tensor{A}$ was characterized in \cite{BV2018}. These theoretical results apply in particular to Grassmann decompositions.

It is natural to wonder about the relation of Grassmann decompositions with respect to other, better-studied rank decompositions. These connections are explored next.

\paragraph{Tensor rank decomposition}
Arguably the most famous rank decomposition is the one with respect to the \emph{Segre variety} $\Var{S}$ \cite{Harris1992,BCCGO2018,Landsberg2012}. This rank decomposition was introduced by Hitchcock \cite{Hitchcock1927} and is variously called the \emph{tensor rank decomposition}, \emph{canonical polyadic decomposition}, \emph{CP decomposition}, \emph{canonical decomposition (CANDECOMP)}, or \emph{parallel factor analysis (PARAFAC)}. The $\Var{S}$-rank of a tensor is usually called \emph{the tensor rank}.
This decomposition has many applications, primarily as a versatile tool for data analysis \cite{SBG2004,BK2025}. 

A Grassmann decomposition $\tensor{A} = \tensor{A}_1 +\dots+ \tensor{A}_r \in \wedge^d \VS{V}$ with $\tensor{A}_i = \vect{v}_i^1 \wedge\dots\wedge \vect{v}_i^d$ can also be expressed as a sum of elementary Segre tensors:
\[
\tensor{A} 
= \sum_{i=1}^d \vect{v}_i^1 \wedge\dots\wedge \vect{v}_i^d
= \frac{1}{d!} \sum_{i=1}^d \sum_{\sigma\in\mathfrak{S}([d])} \sign(\sigma) \vect{v}_i^{\sigma_1} \otimes\dots\otimes \vect{v}_i^{\sigma_d}.
\]
However, this expression is not a tensor rank decomposition of $\tensor{A}$ because it is not of minimal length. The reason is that the tensor rank of an elementary skew-symmetric tensor $\vect{v}_i^1 \wedge\dots\wedge \vect{v}_i^d \simeq \vect{e}_1^\trans{} \wedge\dots\wedge \vect{e}_d^\trans{} \simeq \frac{1}{d!} \det$ is not equal to $d!$. Despite its central role in geometric complexity theory, the precise tensor rank of the determinant $\det$ is not known presently \cite{Landsberg2012,Landsberg2017}. Bounds have been established though. Recently, it was shown that the tensor rank of $\det$ is bounded above by the $d$th Bell number \cite{HGJ2024}, which is strictly less than $d!$ for $d\ge3$. 
In conclusion, there is no clear relation between Grassmann and tensor rank decompositions.

\paragraph{Block term decomposition}
Another well-known rank decomposition is the one with respect to \emph{subspace varieties} \cite{Landsberg2012}, resulting in \emph{block term decompositions} \cite{Lathauwer2008}. 

Recall from \cite{Landsberg2012} that a subspace variety of $\VS{V}_1\otimes\dots\otimes\VS{V}_d$ can be defined as the set 
\[
\Var{S}_{r_1,\dots,r_d} := \{ \tensor{B} \in \VS{W}_1\otimes\dots\otimes\VS{W}_d \mid \VS{W}_i \subset \VS{V}_i \text{ and } \dim \VS{W}_i = r_i, i=1,\dots,d \}.
\]
It is the set of all tensors that are elements of a tensor product subspace $\VS{W}_1\otimes\dots\otimes\VS{W}_d$ where the subspace $\VS{W}_i \subset \VS{V}_i$ has dimension $r_i$. 
These are all tensors whose \emph{multilinear rank} \cite{Hitchcock1927a} is bounded componentwise by $(r_1,\dots,r_d)$. Equivalently, they are tensors whose \emph{Tucker decomposition} \cite{Tucker1966} or \emph{higher-order singular value decomposition} \cite{dLdMV2000} has a core tensor of size no more than $r_1 \times \dots \times r_d$. 

If $V : \FF^{d} \to \VS{V}$ is the linear map that sends the standard basis vector $\vect{e}_i$ of $\FF^d$ to $\vect{v}_i$, then elementary skew-symmetric tensors can be expressed in $\VS{V}\otimes\dots\otimes\VS{V}$ as
\[
\vect{v}_1 \wedge \dots \wedge \vect{v}_d = \frac{1}{d!} \sum_{\sigma\in\mathfrak{S}([d])} \sign(\sigma) \vect{v}_{\sigma_1} \otimes\dots\otimes \vect{v}_{\sigma_d}
= (V \otimes\dots\otimes V)\Bigl( \frac{1}{d!} \tensor{E} \Bigr),
\] 
where $\tensor{E} = \sum_{\sigma\in\mathfrak{S}([d])} \sign(\sigma) \vect{e}_{\sigma_1}\otimes\dots\otimes\vect{e}_{\sigma_d}$ is called the \emph{Levi--Civita tensor} and $V\otimes\dots\otimes V$ is the tensor product of linear maps; see \cref{sec_prelim_tensor}. An elementary skew-symmetric tensor thus admits a Tucker decomposition with factor matrices $(V,\dots,V)$ and the scaled Levi--Civita tensor $\frac{1}{d!} \tensor{E}$ as core tensor. 
It is not difficult to show that the multilinear rank of the latter is $(d,\dots,d)$. Therefore, the elementary skew-symmetric tensors $\Var{G}_{n}^d$ are minimally contained in the subspace variety $\Var{S}_{d,\dots,d} \subset \VS{V}\otimes\dots\otimes\VS{V}$: there is no $k < d$ such that $\Var{G}_n^d \subset \Var{S}_{k,\dots,k}$.

It follows from the foregoing discussion that Grassmann decompositions can be viewed as special, constrained block term decompositions.}

\subsection*{Related algorithms}
To my knowledge, only two concrete algorithms have appeared in the literature dealing with Grassmann decomposition. 

Arrondo, Bernardi, Macias Marquez, and Mourrain \citep{ABMM2019} proposed an extension of the classic apolarity theory \citep{IK1999} to skew-symmetric tensors and used it for Grassmann decomposition of \emph{arbitrary} skew-symmetric tensors in $\wedge^d \VS{V}$ with $d \le 3$ and $n = \dim \VS{V} \le 8$. It is primarily intended for symbolic computations. It can be characterized as a case-by-case analysis based on the existence of certain normal forms. The objective of \changed{the present article} is decidedly more modest: it targets only low-rank, generic tensors. On the other hand, the developed algorithm, \cref{alg_main_computation}, can handle $d=3$ and $n \le 100$, an order-of-magnitude improvement in $n$ over \citep{ABMM2019}.

Recently, Begovi\'c Kova\v{c} and Peri\v{s}a \citep{BP2024} presented a numerical algorithm for the decomposition of Grassmann rank-$1$ skew-symmetric tensors in $\wedge^3 \VS{V}$. It is based on a structure-preserving alternating least-squares approach for tensor rank decomposition. A simpler, noniterative method is presented \changed{for exact decomposition} en route to \cref{alg_main_computation}.

The algorithm developed in this article for Grassmann decomposition, \cref{alg_main_computation}, arose from my understanding of Brooksbank, Kassabov, and Wilson's framework \citep{BKW2024}, based on the talk of M.~Kassabov at the \emph{2024 Tensors: Algebra--Geometry--Applications} conference.
Brooksbank, Kassabov, and Wilson recently introduced in \citep{BKW2024} a general framework for \emph{sparsification} of arbitrary tensors through multilinear multiplication, i.e., a multilinear transformation to a form with few nonzero entries. Their central idea is a Lie algebra construction called a \emph{chisel}, which describes a generalized differentiation \citep{BKW2024}. Notably, a basis of this algebra of chisel derivations can be computed from a system of linear equations that is defined by a multilinear map \citep{BKW2024}. By diagonalizing a generic element of this Lie algebra with an eigenvalue decomposition (EVD) and applying the eigenbases multilinearly to the tensor, a much sparser form is attainable, which depends on the algebraic structure of the chisel. Hence, by appropriately choosing the chisel, different \emph{sparsity patterns} can be detected using the framework of \citep{BKW2024}.

\subsection*{Contributions}

The main contribution of this article is the introduction of a numerical algorithm, \cref{alg_main_computation}, \changed{and a corresponding efficient Julia implementation,} to decompose a \emph{generic} skew-symmetric tensor of small Gr-rank $r \le \frac{n}{d}$ into its unique Grassmann decomposition \cref{eqn_grassmann_decomp}. 
\changed{The algorithm automatically determines the numerical Gr-rank; it is not required to specify the target rank beforehand.} 

The key ingredient of \cref{alg_main_computation} is an eigenvalue decomposition (EVD) of a matrix that \changed{is an element of} the kernel of a natural multilinear map \changed{associated with the tensor.}
\changed{\Cref{alg_main_computation} is designed for \emph{exact Grassmann decomposition}. It is suitable for numerical tensors in the sense that it can tolerate small model violations originating from roundoff errors, as illustrated in the numerical experiments. By contrast, it is not designed for \emph{Grassmann approximation} problems where there are significant deviations from an exact low-rank Grassmann decomposition.}

The proposed \cref{alg_main_computation} follows the high-level framework of Brooksbank, Kassabov, and Wilson \citep{BKW2024}, with one main conceptual difference: I target an algorithm that is capable of recovering the elementary building blocks of one specific family of tensor decompositions, namely Grassmann decompositions, rather than discovering the sparsity pattern of a tensor under a chosen chisel. That is, in the language of \citep{BKW2024}, \changed{we are} looking for an appropriate chisel that can be used to decompose skew-symmetric tensors into their assumed low-rank Grassmann decomposition. The main contribution of this article can thus be characterized alternatively as showing that the \emph{universal chisel} \citep{BKW2024} uncovers Grassmann decompositions of generic low Gr-rank skew-symmetric tensors.

\subsection*{Outline}
\changed{\Cref{sec_preliminaries} recalls standard results from the literature on tensors and the identifiability of Grassmann decompositions. To illustrate Brooksbank, Kassabov, and Wilson's framework \citep{BKW2024} in a more familiar setting, we present} a simple but non-competitive algorithm for tensor rank decomposition \citep{Hitchcock1927a} in \cref{sec_dleto}. Then, \cref{sec_algorithm} develops the main ingredients that constitute the mathematical \cref{alg_main_computation} for low-rank Grassmann decomposition. An efficient numerical algorithm fleshing out the nontrivial technical details of \cref{alg_main_computation} is presented in \cref{sec_numerical_algorithm}. Numerical experiments are featured in \cref{sec_numerical_experiments}, illustrating the computational performance and numerical accuracy of the proposed algorithm.
The article is concluded with some final remarks in \cref{sec_conclusions}.

\section{Preliminaries}\label{sec_preliminaries}

\changed{Standard results from the literature on tensors, algebraic geometry, and the identifiability of Grassmann decompositions are presented in the next subsections. The notation will also be fixed.}

\subsection{Linear and multilinear algebra concepts}\label{sec_prelim_tensor}
Throughout this article, $\VS{W}$ denotes an $m$-dimensional vector space over the real $\FF = \mathbb{R}$ or complex field $\FF = \mathbb{C}$. Similarly, $\VS{V}$ denotes an $n$-dimensional space. 

The dual of a vector space $\VS{V}$ is denoted by $\VS{V}^*$. It is the vector space of linear forms in $\VS{V}$. Throughout the article, the nondegenerate bilinear form $\VS{V}\times\VS{V}\to\FF, (\vect{v},\vect{w})\mapsto  \vect{v}^\trans{}\vect{w} = \sum_{i=1}^n v_i w_i$ is used to identify $\VS{V}$ with $\VS{V}^*$. The matrix space $\VS{V}\otimes\VS{W}^*$ is the linear space of linear maps from $\VS{W}$ to $\VS{V}$.

When discussing metric properties such as orthogonality and approximations, we assume that the vector space $\VS{W}$ is equipped with the standard Frobenius inner product $\langle \vect{x}, \vect{y} \rangle_F = \vect{x}^\herm \vect{y} = \sum_{i=1}^m \overline{x}_i y_i$, where the overline denotes the complex conjugation and $\cdot^\herm$ is the conjugate transpose. 
For real vector spaces $\vect{x}^\herm$ simplifies to $\vect{x}^\trans{}$. 
The induced Frobenius norm is denoted by $\|\vect{x}\|_F$.

\changed{The trace of a linear operator $A : \VS{V} \to \VS{V}$ is denoted by $\mathrm{tr}(A)$ and is defined as the sum of the eigenvalues of $A$. Alternatively, if the matrix $A' = [a_{ij}']$ represents the linear operator $A$ in coordinates with respect to an arbitrary basis, then the trace also equals the sum of the diagonal elements of $A'$: $\mathrm{tr}(A)=\mathrm{tr}(A')=a_{11}' + a_{22}' + \dots + a_{mm}'$, where $m=\dim\VS{V}$.}

The set of all partitions of cardinality $k$ of a set $S$ is denoted by $\binom{S}{k}$.

The vector space of skew-symmetric tensors in $\VS{V}^{\otimes d} := \VS{V}\otimes\dots\otimes\VS{V}$ is denoted by $\wedge^d \VS{V}$. The order \changed{is assumed to satisfy} $3 \le d \le n$, as the other cases are not interesting. Indeed, if $d=1$, then $\wedge^1 \VS{V} = \VS{V}$ and every nonzero vector has Gr-rank $1$.
If $d=2$, then $\wedge^2 \VS{V}$ is the space of skew-symmetric matrices and there are $\infty$-many rank-$r$ Grassmann decompositions for $r \ge 2$ \citep[Remark V.2.9]{Zak1993}. If $d > n$, then $\wedge^d \VS{V} = \{0\}$.

The tensor product of linear maps $A_i : \VS{V} \to \VS{W}$ is the unique linear map \citep[Section 1.16]{Greub1978} with the property that
\[
A_1 \otimes\dots\otimes A_d : \VS{V}^{\otimes d} \to \VS{W}^{\otimes d}, \quad
\vect{v}_1\otimes\dots\otimes\vect{v}_d \mapsto (A_1 \vect{v}_1)\otimes\dots\otimes(A_d \vect{v}_d).
\]
Applying it to a tensor is called a \emph{multilinear multiplication}.
\changed{This operation will be abbreviated to}
\begin{align*}
(A_1,\ldots,A_d) \cdot\tensor{A} &:= (A_1\otimes\dots\otimes A_d)(\tensor{A}),\\
A_k \cdot_k \tensor{A} &:= (\Id_\VS{V},\dots,\Id_\VS{V},A_k,\Id_\VS{V},\dots,\Id_\VS{V}) \cdot \tensor{A},
\end{align*}
where $\Id_\VS{V} : \VS{V} \to \VS{V}$ is the identity map.

Let $\sigma \sqcup \rho$ partition $[d]$ with the cardinality of $\sigma$ being $\sharp\sigma=k$. Then, the $(\sigma;\rho)$-\emph{flattening} \citep{Landsberg2012,Hackbusch2019} of a tensor is the linear isomorphism
\begin{align*}
 \cdot_{(\sigma;\rho)} : \qquad\qquad \VS{V}^{\otimes d} &\to \VS{V}^{\otimes k} \otimes (\VS{V}^{\otimes (d-k)})^*,\\
  \vect{v}_1\otimes\dots\otimes\vect{v}_d &\mapsto (\vect{v}_{\sigma_1}\otimes\dots\otimes \vect{v}_{\sigma_{k}}) (\vect{v}_{\rho_1} \otimes\dots\otimes \vect{v}_{\rho_{d-k}} )^\trans{}.
\end{align*}
The standard flattening $\cdot_{(k; 1,\ldots,k-1,k+1,\ldots,d)}$ will be abbreviated to $\cdot_{(k)}$. Similarly, for $k\ne \ell$, $\cdot_{(k,\ell)} = \cdot_{(k,\ell;\rho)}$ where $\rho$ is an increasingly sorted vector of length $d-2$ whose elements are $[d]\setminus\{k,\ell\}$.

The \emph{multilinear} or \emph{multiplex} rank \citep{Hitchcock1927} of a tensor $\tensor{A}$ is defined as the tuple of ranks of the standard flattenings of $\tensor{A}$:
\(
 \mathrm{mrank}(\tensor{A}) = \bigl( \rank(\tensor{A}_{(1)}),\, \ldots,\, \rank(\tensor{A}_{(d)}) \bigr).
\)

\subsection{Algebraic geometry concepts}
The following results are standard, and can be found, for example, in \citep{CLO2010,Harris1992}. \changed{The material in this subsection is not crucial for understanding this article; however, it is needed to fix the meaning of ``generic'' formally.}

An \emph{algebraic subvariety} of a vector space $\VS{V}$ is the set of points which are the common solutions of a system of polynomial equations, i.e., $\Var{V} = \{ x \in \VS{V} \mid p(x) = 0, \forall p \in I \}$, where $I$ is an ideal of polynomials on $\VS{V}$ defining the variety.
A variety $\Var{V}$ is \emph{irreducible} if $\Var{V}=\Var{U}\cup\Var{W}$ implies that either $\Var{U}\subset\Var{W}$ or $\Var{W}\subset\Var{U}$.
A \emph{Zariski open subset} of $\Var{V}$ is the complement of a \emph{Zariski closed subset} of $\Var{V}$. The Zariski closed subsets of an irreducible variety $\Var{V}$ are all the strict algebraic subvarieties $\Var{W} \subsetneq \Var{V}$. Zariski closed subsets are very small; in particular, a Zariski closed subset of $\Var{V}$ has zero Lebesgue measure on $\Var{V}$. \changed{The Zariski closure of a set $S \subset \Var{V}$ is the smallest algebraic subvariety of $\Var{V}$ that contains $S$. The \emph{dimension} of an irreducible variety $\Var{V}$ is the number $d$ in the longest chain of successively strictly nested Zariski closed subsets:
\[
\emptyset \subsetneq \Var{V}_0 \subsetneq \Var{V}_1 \subsetneq \dots \subsetneq \Var{V}_{d-1} \subsetneq \Var{V}_d=\Var{V};
\]
the empty set has dimension $-1$ by convention, and every finite set of points has dimension $0$.}

\changed{A property $P$ of elements of a variety $\Var{V}$ is called \emph{generic}} if the property fails at most in a strict Zariski closed subset of $\Var{V}$. That is, if $x \in \Var{V}$ does not satisfy property $P$, then there must be nontrivial polynomial equations on $\Var{V}$ defining the failure of property $P$ that $x$ satisfies. For example, a generic matrix $A \in \mathbb{C}^{n\times n}$ is invertible, because singular matrices satisfy the nontrivial polynomial equation $\det(A)=0$. \changed{``Invertibility'' is thus a generic property of the variety of $m\times m$ complex matrices. The word ``generic'' is used in this article exclusively} in the foregoing mathematically precise way.
\changed{To prove a property $P$ is generic on a variety $\Var{V}$, it suffices to show that (i) $x\in\Var{V}$ does not have property $P$ if and only if there exists an ideal of polynomials $I$ such that $p(x) = 0$ for all $p \in I$, and (ii) there exists an $x \in \Var{V}$ with property $P$.}

\changed{A subvariety $\Var{V} \subset \VS{V}$ is a geometric object. At a generic point $x \in \Var{V}$ of an irreducible variety of dimension $d$ there exists a $d$-dimensional affine subspace of $\VS{V}$ with origin at $x$ generated by the tangent vectors of all smooth algebraic curves passing through $x$ in $\VS{V}$. This space is called the \emph{tangent space} to $\Var{V}$ at $x$ and is denoted by $\Tang_{x}{\Var{V}}$. That is, in a formula:
\[
\Tang_x \Var{V} := \mathrm{span}( \left\{ \gamma' (0) \mid \gamma(t)\subset\Var{V} \text{ is smooth algebraic curve with } \gamma(0)=x  \right\}),
\]
where $\gamma'$ denotes the derivative of the curve $\gamma$ with respect to its parameter.}

\subsection{Identifiability of rank decompositions}
\changed{As mentioned previously,} decomposition \cref{eqn_grassmann_decomp} is an instance of a well-studied class of (tensor) rank decompositions \citep{Landsberg2012}. 
\changed{Let $\Var{V} \subset \VS{V}$ be an irreducible variety, and consider the addition map
\begin{equation}\label{eqn_jdp}
 \Sigma_r : (\Var{V}\times\dots\times\Var{V})/\mathfrak{S}([r]) \to \VS{V}, \quad \{\tensor{A}_1,\dots,\tensor{A}_r\} \mapsto \tensor{A}_1 + \dots + \tensor{A}_r.
\end{equation}
We are interested in the \emph{inverse problem} associated with $\Sigma_r$. That is, given a point $\tensor{A}$ in the image of $\Sigma_r$, determine a decomposition in the preimage $\Sigma_r^{-1}(\tensor{A})$. 
A fundamental question concerning this inverse problem is whether it is \emph{well posed}: is there a unique, continuous solution map $\Sigma_r^{-1}$?}

\changed{Well-posedness of inverse problems is considered a necessary requirement for its numerical resolution \cite{Kirsch2011}. Fortunately, the above question has been extensively studied for general rank decompositions with respect to a variety $\Var{V}$.}
The Zariski closure of $\mathrm{Im}(\Sigma_r)$ is the algebraic variety $\sigma_r(\Var{V})$, called the $r$th \emph{secant variety} of the (affine cone over) the variety $\Var{V}$ \citep{BCCGO2018}. The dimension of $\sigma_r(\Var{V})$ is \emph{expected} to equal the minimum of the dimensions of its domain and codomain, i.e., $\min\{ r \dim\Var{V}, \dim \VS{V} \}$. The dimension of $\sigma_r(\Var{V})$ has been relatively well studied during the past two decades \changed{for several varieties $\Var{V}$}.
\changed{If the dimension of $\sigma_r(\Var{V})$} coincides with the dimension of the domain \changed{of $\Sigma_r$, then} this implies that a \emph{generic} element in $\sigma_r(\Var{V})$ has a \emph{finite} number of $\Var{V}$-rank decompositions in its preimage \cite{Landsberg2012,Harris1992}. 
Moreover, \changed{if $\Var{V}$} is a smooth variety, then Massarenti and Mella's wonderful characterization \citep[Theorem 1.5]{MM2024} \changed{implies that the generic fiber of $\Sigma_r$ is a singleton. That is, $\Var{V}$ is generically identifiable. In the case where $\Var{V}$ is the Grassmannian \cite[Lecture 6]{Harris1992}, we have the following result.}

\begin{proposition}\label{prop_wellposed}
\changed{Let $\Var{G}_n^d \subset \VS{V}$ be a Grassmannian with $3 \le d \le n$. If}
  \[
  r < \frac{\dim \wedge^d \VS{V}}{\dim \Var{G}_n^d} - \dim \Var{G}_n^d,
  \]
  then there exists a Zariski open subset $\Var{V} \subset \sigma_r(\Var{G}_n^d)$ such that $\Sigma_r^{-1} : \Var{V} \to (\Var{G}_n^d)^{\times r}/\mathfrak{S}([r])$ is a continuous inverse map of $\Sigma_r$.
\end{proposition}
\begin{proof}
Modulo a few exceptional cases, the $r$th secant variety $\sigma_r(\Var{G}_n^d)$ has the expected dimension for either sufficiently small or large $r$ \citep{CGG2004,Boralevi2013,AOP2011,BDdG2007,TBC2023}. Specifically, Blomenhofer and Casarotti's \citep[Theorem 4.3]{TBC2023} shows that if the Gr-rank of $\tensor{A}$ is less than or equal to
\(
r_\star := \frac{\dim\wedge^d \VS{V}}{\dim \Var{G}_n^d} - \dim \Var{G}_n^d,
\)
then $\dim\sigma_r(\Var{G}_n^d) = r\dim\Var{G}_n^d$. \changed{Since Grassmannians are smooth,} the second part of \citep[Theorem 4.3]{TBC2023} is obtained: A generic Gr-rank-$r$ skew-symmetric tensor $\tensor{A}$ is \emph{identifiable} if $r < r_\star$. 
\changed{Continuity of the inverse map follows from the inverse function theorem; see, e.g., \cite{BV2018}.}
\end{proof}

\section{Chiseling algorithms to detect sparsity patterns in tensor data}\label{sec_dleto}

\changed{Brooksbank, Kassabov, and Wilson \citep{BKW2024} introduced an innovative, general sparsification framework that has the capacity to uncover hidden \emph{sparsity patterns}, i.e., patterns of numerical zeros, in tensors. The motivation underlying their method is the general algebraic principle that the symmetries of a mathematical object, such as a tensor, under transformations, such as group actions, encode its essential properties. Following this guiding principle, \citep{BKW2024} proposes a method to determine the \emph{infinitesimal stabilizer} of a tensor under multilinear multiplication. Reexpressing the tensor in bases corresponding to the invariant subspaces associated with the infinitesimal stabilizer (an element of the \emph{Lie algebra}) will then reveal specific sparsity patterns, depending on the chosen group acting on the tensor. The key details of this framework can be found in \citep{BKW2024}, with further supporting theory in the references of that work. At a high level, to sparsify a tensor $\tensor{A}\in\FF^{r\times r\times r}$, the essential steps of the BKW framework \cite{BKW2024} are as follows:
\begin{enumerate}
\item Choose a suitable linear map $\delta_\tensor{A} : \FF^{r\times r}\times\FF^{r\times r}\times\FF^{r\times r} \to \FF^{r\times r\times r}$.
\item Choose a suitable element $(\dot{X},\dot{Y},\dot{Z})$ from $\ker \delta_\tensor{A} := \{ (X',Y',Z') \mid \delta_\tensor{A}(X',Y',Z') = 0 \}$.
\item Compute the eigendecompositions $\dot{X}=X \Lambda_1 X^{-1}$, $\dot{Y}=Y\Lambda_2Y^{-1}$, and $\dot{Z}=Z\Lambda_3Z^{-1}$.
\item Sparsify the tensor by computing $(X^{-1}, Y^{-1}, Z^{-1}) \cdot \tensor{A}$.
\end{enumerate}}

\changed{The key insight of the present article is that} Brooksbank, Kassabov, and Wilson's sparsification framework \citep{BKW2024} \changed{can be} naturally adopted for the task of decomposing a tensor into elementary tensors. 
\changed{One novel theoretical contribution over \citep{BKW2024} is characterizing under which conditions their framework will correctly recover Grassmann decompositions.}
Before \changed{discussing the Grassmann case, however, we investigate} a simple algorithm for the tensor rank decomposition \changed{of general tensors} to illustrate the main ideas of Brooksbank, Kassabov, and Wilson's sparsification algorithm in a more familiar setting. \changed{My discussion will} use mostly elementary multilinear algebra arguments and avoids the Lie algebra jargon. \changed{To my knowledge, this algorithm was not previously known for tensor rank decomposition of low-rank tensors.}

Assume \changed{for the remainder of this section} that we are given a generic \changed{$r\times r\times r$ tensor of multilinear rank $(r,r,r)$ and tensor rank $r$:}
\begin{equation}\label{eqn_generic_A}
\tensor{A} 
= \sum_{i=1}^r \vect{a}_i \otimes \vect{b}_i \otimes \vect{c}_i
=: \text\textlbrackdbl A, B, C \text\textrbrackdbl, 
\end{equation}
where $A = [\vect{a}_1 \dots \vect{a}_r]$ and likewise for $B$ and $C$ are the \emph{factor matrices}. 
Note the constraint on the ranks and the dimensions of the tensor space: \changed{we assumed essentially that $A$, $B$, and $C$ are invertible matrices.} As per the usual \changed{identifiability} arguments \cite[Theorem 4.1]{COV2014}, the algorithm also applies to tensors which are Tucker compressible to this shape. 
\changed{Observe that in the above shorthand notation, we have the identity
\(
(X, Y, Z) \cdot \text\textlbrackdbl A, B, C \text\textrbrackdbl
= \text\textlbrackdbl X A, Y B, Z C \text\textrbrackdbl
\).}
The next thing we need to consider is 
 
Consider the natural action of the $r\times r$ invertible matrices $\mathrm{GL}(\FF^r)$ on the fixed tensor $\tensor{A}$:
\[
 \phi_\tensor{A} : \mathrm{GL}(\FF^r) \times \mathrm{GL}(\FF^r) \times \mathrm{GL}(\FF^r) \to \FF^{r} \otimes \FF^r\otimes\FF^r, \quad (X,Y,Z) \mapsto (X,Y,Z)\cdot\tensor{A}.
\]
Its differential \changed{at $(I_r,I_r,I_r)$, where $I_r$ is the $r\times r$ identity matrix,} is the \changed{linear} map
\begin{align}\label{eqn_full_dleto}
\begin{split}
\delta_\tensor{A} : \FF^{r\times r} \times \FF^{r\times r} \times \FF^{r\times r}  &\to \FF^r\otimes\FF^r\otimes\FF^r,\\
\quad 
(\dot{X},\dot{Y},\dot{Z}) &\mapsto \dot{X}\cdot_1\tensor{A} + \dot{Y}\cdot_2\tensor{A} + \dot{Z}\cdot_3\tensor{A}.
\end{split}
\end{align}
This differential corresponds to the \emph{universal chisel} in \citep{BKW2024}, \changed{and is the ``suitable linear map'' $\delta_\tensor{A}$ in step $1$ of the BKW framework.}

\changed{Next, we make a crucial observation:} the kernel of $\delta_\tensor{A}$ contains at least the following $2r$-dimensional \changed{linear} subspace \changed{of $\FF^{r\times r}\times\FF^{r\times r}\times\FF^{r\times r}$:}
\begin{equation}\label{eqn_kernel_trd_dleto}
K= \{  (A \diag(\boldsymbol\alpha) A^{-1},\, B \diag(\boldsymbol\beta) B^{-1}, C \diag(\boldsymbol\gamma) C^{-1}) \mid \boldsymbol\alpha+\boldsymbol\beta+\boldsymbol\gamma = 0 \in \FF^{r} \},
\end{equation}
\changed{where $A$, $B$, and $C$ are factor matrices of $\tensor{A}$.}
Indeed, we have by elementary computations that
\begin{align*}
&\delta_{\tensor{A}}(A \diag(\boldsymbol\alpha) A^{-1}, B \diag(\boldsymbol\beta) B^{-1}, C \diag(\boldsymbol\gamma) C^{-1})\\
&\qquad= (A \diag(\boldsymbol\alpha) A^{-1}) \cdot_1 \tensor{A} + (B \diag(\boldsymbol\beta) B^{-1})\cdot_2\tensor{A} + (C \diag(\boldsymbol\gamma) C^{-1})\cdot_3\tensor{A} \\
&\qquad= \text\textlbrackdbl A \diag(\boldsymbol\alpha) A^{-1} A, B, C \text\textrbrackdbl + \text\textlbrackdbl A, B \diag(\boldsymbol\beta) B^{-1} B, C \text\textrbrackdbl + \text\textlbrackdbl A, B, C \diag(\boldsymbol\gamma) C^{-1} C \text\textrbrackdbl \\
&\qquad= \text\textlbrackdbl A \diag(\boldsymbol\alpha), B, C \text\textrbrackdbl + \text\textlbrackdbl A, B \diag(\boldsymbol\beta), C \text\textrbrackdbl + \text\textlbrackdbl A, B, C \diag(\boldsymbol\gamma)\text\textrbrackdbl \\
&\qquad= \text\textlbrackdbl A \diag(\boldsymbol{\alpha} + \boldsymbol\beta + \boldsymbol{\gamma}), B, C \text\textrbrackdbl \\
&\qquad= \text\textlbrackdbl 0, B, C \text\textrbrackdbl\\
&\qquad= 0.
\end{align*}
\changed{This establishes that $K \subset \ker \delta_\tensor{A}$. In fact, we can prove the following result.}

\begin{lemma}
\changed{Let $\tensor{A}$ be as in \cref{eqn_generic_A}, $\delta_\tensor{A}$ as in \cref{eqn_full_dleto}, and $K$ as in \cref{eqn_kernel_trd_dleto}. Then, 
\begin{enumerate}
\item the kernel of $\delta_\tensor{A}$ is \(\ker \delta_\tensor{A} = K,\) and
\item in a generic element $(X,Y,Z)\in K$ the matrices $X$, $Y$, and $Z$ have distinct eigenvalues.
\end{enumerate}}
\end{lemma}
\begin{proof}
\changed{We prove the two properties in the next paragraphs.}

  \paragraph{Property 1.}
\changed{As $A$, $B$, and $C$ are invertible factor matrices, we have for arbitrary $r\times r$ matrices $\dot{X}$, $\dot{Y}$, and $\dot{Z}$ that
\begin{align*}
\delta_\tensor{A}(\dot{X} A^{-1},\dot{Y} B^{-1},\dot{Z} C^{-1})
&= \text\textlbrackdbl \dot{X} A^{-1} A, B, C \text\textrbrackdbl + \text\textlbrackdbl A, \dot{Y} B^{-1} B, C \text\textrbrackdbl + \text\textlbrackdbl A, B, \dot{Z} C^{-1} C \text\textrbrackdbl\\
&= \sum_{i=1}^r \bigl(\dot{\vect{x}}_i \otimes \vect{b}_i \otimes \vect{c}_i + \vect{a}_i \otimes \dot{\vect{y}}_i \otimes \vect{c}_i + \vect{a}_i \otimes \vect{b}_i \otimes \dot{\vect{z}}_i\bigr).
\end{align*}
This last formula is a familiar expression for tangent vectors at generic points on the $r$th secant variety of the \emph{Segre variety} $\Var{S}\subset\FF^{r\times r\times r}$ of rank-$1$ tensors; see, e.g., \cite{AOP2009,Landsberg2012}.
Since $\dot{X}$, $\dot{Y}$, $\dot{Z}$ are arbitrary, the image of $\delta_\tensor{A}$ is 
\[
\mathrm{Im}(\delta_\tensor{A}) 
= \Tang_{\vect{a}_1\otimes\vect{b}_1\otimes\vect{c}_1} \Var{S} + \dots + \Tang_{\vect{a}_r\otimes\vect{b}_r\otimes\vect{c}_r} \Var{S}
= \Tang_\tensor{A} \sigma_r(\Var{S}),
\]
where the last equality exploited the genericity of $\tensor{A} = \vect{a}_1\otimes\vect{b}_1\otimes\vect{c}_1 +\dots+ \vect{a}_r\otimes\vect{b}_r\otimes\vect{c}_r$ and Terracini's lemma \cite{Terracini1911}.}
\changed{Hence, $\delta_\tensor{A}$} surjects onto the tangent space of the $r$th secant variety of the Segre variety $\Var{S}$.
The known nondefectivity results, specifically \citep[Proposition 4.3]{AOP2009}, \changed{entail that $\dim \sigma_r(\Var{S}) = \dim \Tang_\tensor{A} \sigma_r(\Var{S}) = r(3r - 2)$. From this we conclude that 
\[
\rank( \delta_\tensor{A} ) = \dim \mathrm{Im}(\delta_\tensor{A}) = r(3r - 2).
\]
As the dimension of the domain $\FF^{r\times r}\times\FF^{r\times r}\times\FF^{r\times r}$ is $3r^2$, it follows that the kernel of $\delta_\tensor{A}$ is of dimension $2r$. The proof is concluded by the observation that $K \subset \ker \delta_\tensor{A}$ is a $2r$-dimensional subspace, hence we must have equality.}

\paragraph{Property 2.}
\changed{Consider an arbitrary element $(X,Y,Z) \in K$. Then, $X$ has coinciding eigenvalues if and only if the discriminant of the characteristic polynomial is zero \citep[Chapter 12, Section 1.B]{GKZ1994}. Since this discriminant is a polynomial, $X$ having a coinciding eigenvalue occurs only on a Zariski closed subset of the vector space $K$. The analogous observations hold for $Y$ and $Z$. Taking the intersection of these three Zariski closed subsets yields a Zariski closed subset $\Var{Z}$ of $K$ where $X$, $Y$, or $Z$ has some coinciding eigenvalues. 

It only remains to show that $\Var{Z}$ is a \emph{strict} Zariski closed subset of the vector space $K$. For this, it suffices to present one example of an element in $K$ that is not an element of $\Var{Z}$. To this end, take
\[
\boldsymbol\alpha' = (1,2,\dots,r),\, \boldsymbol\beta'=(1,2,\dots,r),\, \text{ and } \boldsymbol\gamma'=(-2,-4,\dots,-2r).
\]
Then $(A\diag(\boldsymbol\alpha')A^{-1}, B\diag(\boldsymbol\beta')B^{-1},C\diag(\boldsymbol\gamma')C^{-1}) \in K \setminus\Var{Z}$. This proves that $\Var{Z}$ is a strict subvariety of $K$. Hence, having distinct eigenvalues is a generic property in $K$.}
\end{proof}

\changed{The previous result entails that all generic elements of $K = \ker \delta_\tensor{A}$ are of the form 
\[
(\dot{X},\dot{Y},\dot{Z}) = (A \diag(\boldsymbol\alpha)A^{-1}, B\diag(\boldsymbol\beta)B^{-1}, C\diag(\boldsymbol\gamma)C^{-1}),\quad
\text{where } {\boldsymbol\alpha}+{\boldsymbol\beta}+{\boldsymbol\gamma}=0,
\]
and all $\alpha_1,\dots,\alpha_r$ are distinct, and likewise for all $\beta_1,\dots,\beta_r$ and all $\gamma_1,\dots,\gamma_r$.
If we take such a generic element $(\dot{X},\dot{Y},\dot{Z})$ in the kernel of $\delta_\tensor{A}$, corresponding to the ``suitable element'' in step $2$ of the BKW framework, then we can compute the eigendecompositions of $\dot{X}$, $\dot{Y}$, and $\dot{Z}$:
\begin{equation}\label{eqn_trd_eigdecs}
\dot{X} = \widetilde{A} \diag(\widetilde{\boldsymbol\alpha}) \widetilde{A}^{-1},\quad 
\dot{Y} = \widetilde{B} \diag(\widetilde{\boldsymbol\beta}) \widetilde{B}^{-1},\quad 
\dot{Z} = \widetilde{C} \diag(\widetilde{\boldsymbol\gamma}) \widetilde{C}^{-1},
\end{equation}
which is the third step in the BKW framework.
As $\dot{X}$, $\dot{Y}$, and $\dot{Z}$ are diagonalizable matrices with distinct eigenvalues, each one of them has a unique set of distinct eigenvalues and a unique set of corresponding one-dimensional invariant subspaces \cite{HJ2013}.
Consequently, there exist \emph{permutation matrices}, i.e., matrices whose columns are a permutation of the identity matrix, $P_1$, $P_2$, and $P_3$, and vectors $\boldsymbol\alpha', \boldsymbol\beta', \boldsymbol\gamma' \in (\FF\setminus\{0\})^r$ such that 
\[
\widetilde{A} = A \diag(\boldsymbol\alpha')^{-1} P_1^\trans{},\quad 
\widetilde{B} = B \diag(\boldsymbol\beta')^{-1} P_2^\trans{},\quad 
\widetilde{C} = C \diag(\boldsymbol\gamma')^{-1} P_3^\trans{}.
\]}

\changed{The final step in the BKW framework consists of} multilinearly multiplying the tensor $\tensor{A}=\text\textlbrackdbl A, B, C\text\textrbrackdbl$ \changed{by the inverses of $\widetilde{A}$, $\widetilde{B}$, and $\widetilde{C}$:
\begin{align*}
  \tensor{S} := 
(\widetilde{A}^{-1}, \widetilde{B}^{-1}, \widetilde{C}^{-1})\cdot\tensor{A}
&= (\widetilde{A}^{-1}, \widetilde{B}^{-1}, \widetilde{C}^{-1})\cdot\text\textlbrackdbl A, B, C \text\textrbrackdbl\\
&= \text\textlbrackdbl (A \diag(\boldsymbol\alpha')^{-1} P_1^\trans{})^{-1} A, (B \diag(\boldsymbol\beta')^{-1} P_2^\trans{})^{-1} B, (C \diag(\boldsymbol\gamma')^{-1} P_3^\trans{})^{-1} C \text\textrbrackdbl\\
&= \text\textlbrackdbl P_1 \diag(\boldsymbol\alpha'), P_2 \diag(\boldsymbol\beta'), P_3 \diag(\boldsymbol\gamma') \text\textrbrackdbl.
\end{align*}
Since the $P_i$'s are permutation matrices, there exist permutations $\pi_i$ of $[r]$ such that the $j$th column of $P_i$ is $\vect{e}_{\pi_i(j)}$. Consequently,} a sparse tensor is obtained that contains only $r$ nonzero elements:
\changed{\begin{equation}\label{eqn_trd_sparsified}
\tensor{S} = \sum_{i=1}^r (\alpha_i' \beta_i' \gamma_i')\cdot \vect{e}_{\pi_1(i)}\otimes\vect{e}_{\pi_2(i)}\otimes\vect{e}_{\pi_3(i)}.
\end{equation}
The nonzero elements of $\tensor{S}$ appear at the indices $(\pi_1(i),\pi_2(i),\pi_3(i))$ for $i=1,\dots,r$. Hence, based on the positions of the $r$ numerically nonzero elements of $\tensor{S}$, we can determine the permutation matrices $P_1$, $P_2$, and $P_3$. This is important because they determine which columns of $\widetilde{A}$, $\widetilde{B}$, and $\widetilde{C}$ belong together. Indeed, from \cref{eqn_trd_sparsified} we find
\[
\tensor{A} 
= (\widetilde{A},\widetilde{B},\widetilde{C})\cdot\tensor{S} 
= \sum_{i=1}^r (\alpha_i' \beta_i' \gamma_i')\cdot \widetilde{\vect{a}}_{\pi_1(i)} \otimes \widetilde{\vect{b}}_{\pi_2(i)}\otimes \widetilde{\vect{c}}_{\pi_3(i)}.
\]
Moreover, the coefficients $(\alpha_i' \beta_i' \gamma_i')$ can be found as the entry $s_{\pi_1(i),\pi_2(i),\pi_3(i)}$ of $\tensor{S}$ because of \cref{eqn_trd_sparsified}. 
Hence, through the BKW framework we can compute from the input tensor $\tensor{A}$ the set
\[
\{ s_{\pi_1(i),\pi_2(i),\pi_3(i)} \widetilde{\vect{a}}_{\pi_1(i)}\otimes\widetilde{\vect{b}}_{\pi_2(i)}\otimes\widetilde{\vect{c}}_{\pi_3(i)} \mid i=1,\dots,r \}
\]
of} rank-$1$ tensors from $\tensor{A}$'s tensor rank decomposition, \changed{where the individual vectors are obtained from \cref{eqn_trd_eigdecs} and their correct permutations and coefficients from \cref{eqn_trd_sparsified}.}

While the above algorithm is valid and to my knowledge novel, it does not appear to offer advantages over pencil-based algorithms, such as \citep{SY1980,Lorber1985,SK1990,LRA1993,FFB2001,DdL2014,DdL2017,TeV2022}. One of the main reasons is that determining the kernel of $\delta_{\tensor{A}}$ is much more expensive than the $\mathcal{O}(r^4)$ cost for a pencil-based algorithm. However, the idea of this algorithm, based on the sparsification algorithm of \citep[Section 5]{BKW2024}, transfers to \changed{other rank decompositions as well}.

\section{An algorithm for Grassmann decomposition}\label{sec_algorithm}
Returning to the main setting, \changed{we can} apply the template from \cref{sec_dleto}, which follows the framework in \citep[Section 5]{BKW2024}, to design an algorithm for Grassmann decomposition.

\subsection{Reduction to concise spaces}
As in \cref{sec_dleto}, the main strategy applies to tensors that are \emph{concise}, i.e., tensors whose components of the multilinear rank coincide with the dimension of the corresponding vector space. We show in this subsection that a Grassmann decomposition of a nonconcise tensor $\tensor{T}$ can be obtained from a Grassmann decomposition of the \emph{Tucker compression} \citep{Tucker1966} of $\tensor{T}$ to a concise tensor space.

First, two results about flattenings and the generic multilinear rank of skew-symmetric tensors \changed{are presented}, which are certainly known to the experts even though I could not locate a precise reference in the literature. \changed{Some explicit results on the multilinear rank of skew-symmetric tensor are presented in \cite[Section 2]{BK2017}.}
Flattenings of skew-symmetric tensors can be characterized as follows.

\begin{lemma}[Flattening] \label{lem_wedge_flattening}
Let $\sigma \sqcup \rho = [d]$ with $\sharp\sigma=k$ be a partition, and let $\pi=(\sigma, \rho)$ denote the concatenation of $\sigma$ and $\rho$.
Then, we have
\[
 (\vect{v}_1 \wedge\dots\wedge \vect{v}_d)_{(\sigma;\rho)}
 = \sign(\pi) \binom{d}{k}^{-1} \sum_{\eta\in\binom{[d]}{k}} (\vect{v}_{\eta_1}\wedge\dots\wedge \vect{v}_{\eta_k}) (\vect{v}_{\theta_1} \wedge\dots\wedge \vect{v}_{\theta_{d-k}})^\trans{},
\]
where $\theta = \changed{\eta^\perp} := [d]\setminus\eta$, sorted increasingly,
which \changed{is an element of} $(\wedge^k \VS{V}) \otimes (\wedge^{d-k} \VS{V})^*$.
\end{lemma}
\begin{proof}
This is a straightforward computation. Indeed,
\begin{align*}
(\vect{v}_1 \wedge\dots\wedge \vect{v}_d)_{(\sigma;\rho)}
&= \sign(\pi) (\vect{v}_1 \wedge\dots\wedge \vect{v}_d)_{(1,\ldots,k;k+1,\ldots,d)}\\
&= \frac{1}{d!} \sign(\pi) \sum_{\varsigma\in\mathfrak{S}([d])} (\vect{v}_{\varsigma_1}\otimes\dots\otimes\vect{v}_{\varsigma_k})(\vect{v}_{\varsigma_{k+1}}\otimes\dots\otimes\vect{v}_{\varsigma_d})^\trans{}\\
&= \frac{1}{d!} \sign(\pi) \sum_{\eta \in \binom{[d]}{k}} \sum_{s \in \mathfrak{S}(\eta)} \sum_{r\in\mathfrak{S}(\eta^\perp)} (\vect{v}_{s_1}\otimes\dots\otimes\vect{v}_{s_k})(\vect{v}_{r_{1}}\otimes\dots\otimes\vect{v}_{r_{d-k}})^\trans{}\\
&= \frac{1}{d!} \sign(\pi) \sum_{\eta \in \binom{[d]}{k}} \sum_{s \in \mathfrak{S}(\eta)} \left( (\vect{v}_{s_1}\otimes\dots\otimes\vect{v}_{s_k}) \sum_{r\in\mathfrak{S}(\eta^\perp)} (\vect{v}_{r_{1}}\otimes\dots\otimes\vect{v}_{r_{d-k}})^\trans{} \right),
\end{align*} 
where the first equality is \cref{lem_elementary_props}(2) and the last equality exploited the bilinearity of the tensor product of two factors. In the final expression, by once more exploiting \cref{eqn_wedge_product}, we quickly recognize the scaled wedge products from the statement of lemma.
This concludes the proof.
\end{proof}

\changed{The previous result defines flattenings for elementary skew-symmetric tensors as a map from $\Var{G}_n^d \to (\wedge^k \VS{V})\otimes(\wedge^{d-k} \VS{V})^*$. By \cref{lem_basis}, every skew-symmetric tensor can be written uniquely as a linear combination of elementary skew-symmetric tensors. Therefore, the map defined in \cref{lem_wedge_flattening} can be extended linearly to a linear map $\cdot_{(\sigma;\rho)} : \wedge^d \VS{V} \to (\wedge^k \VS{V})\otimes(\wedge^{d-k} \VS{V})^*$.}

\changed{As was already proved in \cite[Section 2]{BK2017}, it follows from the structure of the (standard) flattenings with $\sharp\sigma = 1$ that the multilinear rank of a skew-symmetric tensor is always of the form $\mathrm{mrank}(\tensor{T})=(k,\dots,k)$ for some integer $k$.}

A standard result about the tensor rank decomposition \citep{Hitchcock1927a} of tensors in $\VS{V}_1 \otimes\dots\otimes \VS{V}_d$ is that rank-$r$ tensors with $r \le \min_i \dim \VS{V}_i$ have multilinear rank $(r,\ldots,r)$ in a Zariski open subset of the algebraic variety that is the (Zariski) closure of the set of rank-$r$ tensors \citep{Landsberg2012}. An analogous statement holds for rank-$r$ Grassmann decompositions in $\wedge^d \VS{V}$.

\begin{lemma}[Multilinear rank]\label{lem_full_mlrank}
Let \(\tensor{T} \in \wedge^d \VS{W}\) be a skew-symmetric tensor of Gr-rank $r$ with $dr \le \dim \VS{W}$. Then, the multilinear rank of $\tensor{T}$ is bounded componentwise by $(dr,\ldots,dr)$. If $\tensor{T}$ is generic, then $\mathrm{mrank}(\tensor{T})=(dr,\dots,dr)$.
\end{lemma}
\begin{proof}
Since a Grassmann decomposition of a tensor is a linear combination of elementary tensors, it follows from the expression in \cref{lem_wedge_flattening} that the rank of $\tensor{A}_{(k)}$ is upper bounded by $rd$.

A tensor whose multilinear rank is strictly less than $rd$ in some factor $k$, i.e., $\rank(\tensor{A}_{(k)}) < rd$, satisfies a system of polynomial equations: all the $rd \times rd$ minors of $\tensor{A}_{(k)}$ vanish in this case. These equations define a strict Zariski closed subvariety, because the tensor
$\tensor{E} = \sum_{i=1}^r \vect{e}_{d(i-1)+1} \wedge\dots\wedge \vect{e}_{d(i-1)+d}$, where \changed{$(\vect{e}_1,\dots,\vect{e}_m)$ is any basis of $\VS{W}$}, does not satisfy it. Indeed, $\frac{1}{(d-1)!} \tensor{E}_{(k)}$ is equal to
\[
 \sum_{i=1}^r \sum_{k=1}^d \epsilon_i \vect{e}_{d(i-1)+k} (\vect{e}_{d(i-1)+1} \wedge\dots\wedge \vect{e}_{d(i-1)+k-1}\wedge\vect{e}_{d(i-1)+k+1} \wedge\dots\wedge \vect{e}_{d(i-1)+d})^\trans{},
\]
where $\epsilon_i \in \{-1,1\}$ are left unspecified. This expression specifies a matrix decomposition of the form $V W^\trans{}$, where $V$ is an $m \times rd$ matrix whose columns contain the first $rd \le m$ basis vectors $\vect{e}_i$ and $W$ is an $\binom{m}{d-1} \times rd$ matrix whose columns contain the wedge products. The matrix $W$ has rank $rd \le m \le \binom{n}{d-1}$ because its columns contain a subset of the basis vectors of $\wedge^{d-1} \VS{W}$. It follows that $\mathcal{E}_{(k)} = VW^\trans{}$ has rank $rd$.
Since the standard skew-symmetric flattenings are all equal up to sign by \cref{lem_wedge_flattening}, this concludes the proof.
\end{proof}
 
\changed{For Gr-rank $1$ we can even be a bit more precise.

\begin{corollary}
The multilinear rank of every elementary skew-symmetric tensor (i.e., Gr-rank $1$) in $\wedge^d \VS{W}$ with $d \le \dim \VS{W}$ is $(d,\dots,d)$.
\end{corollary}
\begin{proof}
  Recall that multilinear rank is invariant under multilinear multiplication with invertible matrices \cite{Landsberg2012}.
Recall furthermore that $\Var{G}_m^d$ is a homogeneous space: for every $\tensor{T} \in \Var{G}_m^d$ there exists an invertible matrix $A \in \FF^{m\times m}$ such that $\tensor{T} = (A,\dots,A) \cdot \tensor{E}$, where $\tensor{E}$ is the tensor from the proof of \cref{lem_full_mlrank} with $r=1$. This homogeneity is an easy consequence of \cref{lem_mult_skew}.
Since the multilinear rank of $\tensor{E}$ is $(d,\dots,d)$ by the proof of \cref{lem_full_mlrank}, this concludes the proof.
\end{proof}}

Interpreted differently, \cref{lem_full_mlrank} states that if $m=\dim \VS{W}$ is large relative to the Gr-rank $r$ of a skew-symmetric tensor $\tensor{T} \in \wedge^d \VS{W}$, then there exists a subspace $\VS{V} \subset \VS{W}$ with $n = \dim \VS{V} \le dr$ such that $\tensor{T} \in \wedge^d \VS{V}$. Similar to the case of tensor rank decompositions, we can look for a Grassmann decomposition of $\tensor{T}$ inside the concise tensor space $\wedge^d \VS{V}$. This is the next standard result.

\begin{lemma}[Compression]\label{lem_compression}
Let $\tensor{T} \in \wedge^d \VS{W}$ be a Gr-rank-$r$ skew-symmetric tensor. If there exists a strict subspace $\VS{V} \subset \VS{W}$ such that $\tensor{T} \in \wedge^d \VS{V}$, then at least one of $\tensor{T}$'s Grassmann decompositions \changed{is contained in this space}:
\[
 \tensor{T}
 = \sum_{i=1}^r \vect{v}_{i}^1 \wedge\dots\wedge \vect{v}_i^d, \quad\text{where } \forall i, k: \vect{v}_i^k \in \VS{V}.
\]
If $\tensor{T}\in\wedge^d \VS{W}$ has a unique Grassmann decomposition, then it is necessarily \changed{an element of} $\wedge^d \VS{V}$.
\end{lemma}
\begin{proof}
Let $P : \VS{W} \to \VS{V}$ be a projection. Since $P^{\otimes d}$ is a projection from $\VS{W}^{\otimes d}$ to $\VS{V}^{\otimes d}$, the subspace $\wedge^d \VS{V} \subset \VS{W}^{\otimes d}$ is preserved under the action of $P^{\otimes d}$. Hence, for every rank-$r$ Grassmann decomposition, we have
\[
 (P\otimes\dots\otimes P)(\tensor{T})
 = \sum_{i=1}^r (P \vect{v}_i^1) \wedge\dots\wedge (P \vect{v}_i^d) = \tensor{T}.
\]
This concludes the proof as $P \vect{v}_i^k \in \VS{V}$.
\end{proof}

The key implication of \cref{lem_compression} is that we can restrict our attention to concise tensor spaces. \changed{If $\tensor{T} \in \wedge^d \VS{W}$ is viewed as a tensor in $\VS{W}\otimes\dots\otimes\VS{W}$, then the T-HOSVD algorithm from \cite{BK2017} can be used to obtain a concise representation of $\tensor{T}$.}

Based on the foregoing observations, we can decompose Gr-rank-$1$ tensors with a simpler method than the ones of \citep{BP2024}.

\begin{lemma}[Decomposing elementary tensors]\label{prop_rank1}
Let $\tensor{T} \in \wedge^d \VS{W}$ be an elementary skew-symmetric tensor. Then, there exists a nonzero scalar $\alpha\in\FF$ so that $\tensor{T} = \alpha \vect{u}_1 \wedge\dots\wedge \vect{u}_d$, where $(\vect{u}_1, \dots, \vect{u}_d)$ is a basis of the column span of $\tensor{T}_{(1)}$.
\end{lemma}
\begin{proof}
If $\tensor{T} = \vect{w}_1\wedge\dots\wedge\vect{w}_d \in \wedge^d \VS{W} \simeq \wedge^d \FF^{m}$, then
\[
 \tensor{T}_{(1)} = \sum_{i=1}^d (-1)^{i-1} \vect{w}_i (\vect{w}_1 \wedge\dots\wedge \vect{w}_{i-1}\wedge\vect{w}_{i+1}\wedge\dots\wedge \vect{w}_d)^\trans{} =: W X^\trans{}.
\]
The matrix $W \in \FF^{m \times d}$ has linearly independent columns, for otherwise $\tensor{T}=0$ by \cref{lem_elementary_props}. The matrix $X \in \FF^{\binom{m}{d-1} \times d}$ has linearly independent columns as well, because they form a subset of the induced basis vectors (see \cref{lem_basis}) of $\wedge^{d-1} \VS{W}$ using any completion of $(\vect{w}_1,\dots,\vect{w}_d)$ to a basis of $\VS{W}$. Such a completion exists because $d \le m$, for otherwise $\tensor{T}=0$ by \cref{lem_elementary_props}(1). Consequently, the column span of $\tensor{T}_{(1)}$ is $\VS{U} = \mathrm{span}(\vect{w}_1,\dots,\vect{w}_d)$.
It follows from \cref{cor_det_transform} that any basis $(\vect{u}_1,\dots,\vect{u}_d)$ of $\VS{U}$ satisfies $\tensor{T} = \alpha \vect{u}_1 \wedge\dots\wedge \vect{u}_d$.
\end{proof}

The nonzero scalar $\alpha\in\FF$ that was left unspecified can be determined by solving a linear system, for example by looking at just one of the tensor's nonzero coordinates.

\subsection{The key ingredients}

I claim that we can recover the Grassmann decomposition of a generic concise tensor $\tensor{A} \in \wedge^d \VS{V}$ whose Gr-rank is equal to $r = \frac{1}{d}\dim \VS{V}$, so $\mathrm{mrank}(\tensor{A})=(dr,\dots,dr)=(n,\dots,n)$, from the kernel of the differential of the multilinear map
\[
 \phi_{\tensor{A}} : \mathrm{Aut}(\VS{V}) \to \wedge^d \VS{V}, \quad X \mapsto (X, \dots, X)\cdot \tensor{A},
\]
where $\mathrm{Aut}(\VS{V}) \simeq \mathrm{GL}(\VS{V})$ is the space of linear automorphisms of $\VS{V}$, i.e., the invertible linear maps from $\VS{V}$ to itself.
In the remainder of this paper, we let
\begin{equation} \label{eqn_delta}
 \delta_\tensor{A} := \der_{\Id_\VS{V}} \phi_\tensor{A} \::\: \mathrm{End}(\VS{V}) \to \wedge^d \VS{V},\quad \dot{A} \mapsto \sum_{k=1}^d \dot{A} \cdot_k \tensor{A},
\end{equation}
where $\mathrm{End}(\VS{V})$ is the space of linear endomorphisms of $\VS{V}$.
Note that $\delta_\tensor{A}$ is the natural symmetric variant of the map \cref{eqn_full_dleto}.
It corresponds to using the symmetrized version of the \emph{universal chisel} \citep[Section 7.1]{BKW2024}, as is hinted at in \citep[Section 8.4]{BKW2024}.

As in \cref{sec_dleto}, we need to determine the kernel of $\delta_{\tensor{A}}$ for the tensor $\tensor{A}$ that we wish to decompose. Note that $\delta_\tensor{A}$ is linear in $\tensor{A}$ and a Grassmann decomposition expresses the latter as a linear combination of elementary tensors. Therefore, it suffices to understand the action of $\delta_{\tensor{A}}$ at an elementary skew-symmetric tensor.

\begin{lemma}[Differential]\label{lem_derivative}
Let $\tensor{A} = \vect{v}_1 \wedge\dots\wedge \vect{v}_d \in \wedge^d \VS{V}$. Then, the derivative of $\phi_\tensor{A}$ at the identity $\mathrm{Id}_\VS{V}$ is
\[
\delta_\tensor{A}(\dot{A}) = \sum_{k=1}^d \vect{v}_1 \wedge\cdots\wedge \vect{v}_{k-1} \wedge (\dot{A} \vect{v}_k) \wedge \vect{v}_{k+1} \wedge\cdots\wedge \vect{v}_d.
\]
\end{lemma}
\begin{proof}
Because of the multilinearity of $\wedge$, by \cref{lem_elementary_props}(3), we have
\begin{align*}
 \phi_{\tensor{A}}(\Id_\VS{V} + \epsilon \dot{A})
 &= \bigl( (\Id_\VS{V} + \epsilon \dot{A}) \vect{v}_1 \bigr) \wedge\dots\wedge \bigl( (\Id_\VS{V} + \epsilon \dot{A}) \vect{v}_d \bigr)\\
 &= \phi_\tensor{A}(\Id_\VS{V}) + \epsilon \cdot \delta_\tensor{A}(\dot{A}) + o(\epsilon).
\end{align*}
The result follows from the definition of the directional derivative and the fact that $\mathrm{Aut}(\VS{V})$ is an open subset of the linear space of endomorphisms $\mathrm{End}(\VS{V})$.
\end{proof}

In the remainder of this text, it will be convenient to parameterize Grassmann decompositions using factor matrices. Let
\[
V_i = \begin{bmatrix} \vect{v}_i^1 & \dots & \vect{v}_i^d \end{bmatrix} \in \FF^{n \times d}
\quad\text{and}\quad
V = \begin{bmatrix} V_1 & \dots & V_r \end{bmatrix} \in \FF^{n \times n}
\]
be, respectively, the \emph{elementary factor matrix} of the $i$th elementary Grassmann tensor $\tensor{A}_i = \vect{v}_i^1 \wedge\dots\wedge \vect{v}_i^d$, and the \emph{decomposition factor matrix} of the rank-$r$ Grassmann decomposition $\tensor{A} = \tensor{A}_1 + \dots + \tensor{A}_r$. Neither elementary nor decomposition factor matrices are unique, given a Grassmann decomposition.
For example, any permutation of the matrices $V_i$ in $V$ would represent the same Grassmann decomposition.
More fundamentally and relevantly, $V_i D_i$ represents the same elementary tensor for all matrices $D_i$ with $\det(D_i)=1$, because of \cref{cor_det_transform}.

Based on the characterization in \cref{lem_derivative}, we can determine the structure of the kernel of $\delta_\tensor{A}$ at a generic low-rank Grassmann decomposition \changed{inside} a sufficiently large concise tensor space.

\begin{theorem}[Kernel structure theorem]\label{thm_main}
Consider a generic skew-symmetric tensor of Grassmann rank $r$,
\[
  \tensor{A} = \sum_{i=1}^r \vect{v}_i^1 \wedge\dots\wedge \vect{v}_i^d \in \wedge^d \VS{V} \simeq \wedge^d \FF^n,
\]
with decomposition factor matrix $V$, $n = dr$, and $d \ge 3$. Then, the kernel of $\delta_\tensor{A}$ is the following $r(d^2-1)$-dimensional linear subspace of $\FF^{n \times n}$:
\begin{align}\label{eqn_kernel}
 \kappa_\tensor{A} := \ker \delta_\tensor{A} =
 \{
 V \mathrm{diag}(A_1, \ldots, A_r) V^{-1} \mid  A_i \in \FF^{d \times d} \text{ with } \mathrm{tr}(A_i) = 0
 \}.
\end{align}
\end{theorem}
\begin{proof}
By \cref{lem_full_mlrank}, we can assume that $\tensor{A}$ has multilinear rank $(dr,\ldots,dr)$. Because of \cref{lem_wedge_flattening} and linearity of the Grassmann decomposition, it then follows that the $\vect{v}_{i}^k$'s form a basis of $\FF^n$. The dual basis vector of $\vect{v}_i^k$ \changed{will be denoted} by $\vect{\partial}_{i}^k$.

\changed{First, we determine a subspace $\VS{K} \subset \mathrm{End}(\VS{V})$ that is contained in the kernel of $\delta_{\tensor{A}}$.}
By \cref{lem_derivative} and linearity, we have 
\[
 \delta_\tensor{A}(\dot{A}) = \sum_{i=1}^r \sum_{k=1}^d \vect{v}_i^1\wedge\dots\wedge\vect{v}_i^{k-1} \wedge (\dot{A} \vect{v}_i^k) \wedge \vect{v}_i^{k+1}\wedge\dots\wedge \vect{v}_i^d.
\]
Then, if $\dot{A}_i^k := (\lambda_i^{1k} \vect{v}_i^1 +\cdots+ \lambda_i^{dk} \vect{v}_i^d) \partial_i^k$, where the superscripts are indices, then we compute that
\begin{align*}
\delta_\tensor{A}(\dot{A}_i^k) 
&= 
\vect{v}_i^1\wedge\dots\wedge\vect{v}_i^{k-1} \wedge ( \lambda_i^{1k} \vect{v}_i^1 +\cdots+ \lambda_i^{dk} \vect{v}_i^d ) \wedge \vect{v}_i^{k+1}\wedge\dots\wedge \vect{v}_i^d\\
&= \lambda_i^{kk} \vect{v}_i^1 \wedge\dots\wedge \vect{v}_i^d,
\end{align*}
because of \cref{lem_elementary_props}(1) and (3). Letting $\dot{A}_i = \dot{A}_i^1+\cdots+\dot{A}_i^d$ then yields
\[
 \delta_\tensor{A}(\dot{A}_i) 
 = (\lambda_i^{11}+\cdots+\lambda_i^{dd}) \vect{v}_i^1 \wedge\dots\wedge \vect{v}_i^d
 = \changed{\mathrm{tr}(\Lambda_i) \vect{v}_i^1 \wedge\dots\wedge \vect{v}_i^d,}
\]
\changed{where $\Lambda_i := [\lambda_{i}^{k\ell}] \in \FF^{d \times d}$. Observe that by definition, 
\[
\dot{A}_i 
= \sum_{k=1}^d \dot{A}_i^k 
= \sum_{k=1}^d \sum_{j=1}^d \lambda_i^{jk} \vect{v}_i^j \partial_i^k,
\]
which is an endomorphism of the subspace $\VS{V}_i := \mathrm{span}(\vect{v}_i^1, \dots, \vect{v}_i^d) \subset \VS{V}$. Since the multilinear rank of $\tensor{A}$ is $(dr,\dots,dr)$ and $\dim \VS{V} = dr$, it follows that we have the direct decomposition of subspaces $\VS{V} = \VS{V}_1 \oplus \VS{V}_2 \oplus \dots \oplus \VS{V}_r$. Consequently, every linear endomorphism $\dot{A} \in \operatorname{End}(\VS{V})$ for which for all $i=1,\dots,r$, the space $\VS{V}_i$ is an invariant subspace and the restriction of $\dot{A}$ to this invariant subspace is traceless, i.e., 
\[
\mathrm{tr}\left( \dot{A}|_{\VS{V}_i} \right) = \mathrm{tr}\left( \dot{A}_i \right) = 0
\]
will be an element of the kernel of $\delta_{\tensor{A}}$. The linear subspace of all these operators is $\VS{K}\subset \ker \delta_{\tensor{A}}$.}
Observe that the dimension of $\VS{K}$ is $r(d^2 -1)$, because \changed{$\VS{V}$ and its decomposition into invariant subspaces} is fixed, and that it coincides with the right-hand side of \cref{eqn_kernel} \changed{under the isomorphism that identifies $\VS{V}$ with $\FF^n$}. Hence, $\dim \kappa_\tensor{A} \ge r(d^2 - 1)$.

\changed{Second,} to show that the kernel is not larger, we proceed as follows. Observe that the affine tangent space to the affine $r$th secant variety $\sigma_r$ of the Grassmannian $\mathrm{Gr}(d,\FF^n)$ at $\tensor{A}$ is, due to Terracini's lemma \citep{Terracini1911} and the genericity of $\tensor{A}$, equal to
\[
\Tang_{\tensor{A}} \sigma_r = \left\{
 \sum_{i=1}^r \sum_{k=1}^d \vect{v}_i^1\wedge\dots\wedge\vect{v}_i^{k-1} \wedge \dot{\vect{w}}_i^k \wedge \vect{v}_i^{k+1}\wedge\dots\wedge \vect{v}_i^d \mid \dot{\vect{w}}_i^k \in \FF^n \right\};
\]
see, e.g., \citep{Boralevi2013}. Let $\dot{W} = [\dot{\vect{w}}_{i}^k]$, and then since $V$ is invertible, we have
\[
 \Tang_{\tensor{A}} \sigma_r = \{ \delta_\tensor{A}(\dot{W}V^{-1}) \mid \dot{W} \in \FF^{n\times n} \} = \mathrm{Im}( \delta_\tensor{A} ).
\]
By the nondefectivity result for $d \ge 3$ in \citep[Theorem 2.1]{CGG2004}, we have
\[
\dim \Tang_{\tensor{A}} \sigma_r = r (1 + \dim \mathrm{Gr}(d, \FF^n)) = r(1 + d(n-d)) = n^2 - r(d^2-1),
\]
having used $n = dr$.
Thus, $\dim \kappa_\tensor{A} \le r(d^2-1)$, which concludes the proof.
\end{proof}

The proof of \cref{thm_main} shows that the structure of the kernel does not depend on the particular decomposition or the \emph{choice} of the $\vect{v}_i^k$'s. Each Grassmann decomposition of $\tensor{A}$ yields an \emph{equivalent description} of the same kernel.
The identifiability of the kernel implies identifiability of the Grassmann decomposition of $\tensor{A}$. This will be shown through the next series of results.

\begin{lemma}[Generic diagonalizability]\label{lem_evd}
Let $\tensor{A}$ and $V$ be as in \cref{thm_main}. 
A generic element $K$ of $\kappa_\tensor{A}$ has distinct eigenvalues and is hence diagonalizable. Moreover, if $K = Z \Lambda Z^{-1}$ is any EVD, then there is a permutation matrix $P$ so that
\(
 \mathrm{span}(V_i) = \mathrm{span}(Z_i'),
\)
where $Z_i' \in \FF^{n \times d}$ and $Z P = \begin{bmatrix}Z_1' &\dots& Z_r'\end{bmatrix}$.
\end{lemma}
\begin{proof}
A matrix $A$ has an eigenvalue of multiplicity $k > 1$ if and only if the discriminant of the characteristic polynomial, a nontrivial polynomial in the entries of $K$, vanishes \citep[Chapter 12, Section 1.B]{GKZ1994}. The matrices
\begin{equation}\label{eqn_good_traceless}
\Delta_k = \diag\left(-(2k+1)\binom{d}{2},kd+1,kd+2,\ldots,kd+d-1\right) 
\end{equation}
have zero trace and no coinciding eigenvalues. Therefore, the diagonal matrix $\Delta = \diag(\Delta_1, \dots, \Delta_r)$ is traceless and has no coinciding eigenvalues. Since $V \Delta V^{-1} \in \kappa_\tensor{A}$ and it has a nonvanishing discriminant, this entails that the variety of matrices in $\kappa_\tensor{A}$ with coinciding eigenvalues is a strict Zariski closed subset. Matrices with distinct eigenvalues are diagonalizable \citep[Theorem 1.3.9]{HJ2013}.

The second part is a corollary of the essential uniqueness of the EVD of diagonalizable matrices, see, e.g., \citep[Theorem 1.3.27]{HJ2013}. In particular, the eigenvectors corresponding to a particular eigenvalue are unique up to scale. Hence, the eigenspace corresponding to some subset of distinct eigenvalues is unique.
\end{proof}

The previous result showed that diagonalization of a generic element in the kernel identifies a set of basis vectors that can be partitioned to provide bases of $\mathrm{span}(\vect{v}_i^1,\ldots,\vect{v}_i^d)$. By \cref{cor_det_transform}, such bases identify the elementary tensors $\vect{v}_i^1 \wedge\dots\wedge \vect{v}_i^d$ up to scale. However, \cref{lem_evd} did not clarify how to perform this partitioning, i.e., how to find $P$. This is covered by the next result.

\begin{lemma}[Generic block diagonalizability]\label{lem_permutation}
Let $\tensor{A}$, $Z$, and $P$ be as in \cref{lem_evd}.
A generic element $K' \in \kappa_\tensor{A}$ is, up to permutation, block diagonalized by $Z$:
\[
Z^{-1} K' Z = P \diag(A_1, \dots, A_r) P^\trans{}, 
\]
where the $A_i \in (\FF\setminus\{0\})^{d\times d}$ have zero trace.
\end{lemma}
\begin{proof}
As $K' \in \kappa_\tensor{A}$, it can be written as $K' = V \diag(A_1',\dots,A_r') V^{-1}$ with $A_i'$ traceless $d\times d$ matrices. Then, \cref{lem_evd} states that there exists a permutation $P$ and $d \times d$ invertible matrices $X_i$ such that
\(
 Z P = V \diag(X_1,\dots,X_r).
\)
Hence, we find that
\begin{align*}
 Z^{-1} K' Z 
 &= P \diag(X_1^{-1},\dots,X_r^{-1}) V^{-1} K' V \diag(X_1,\dots,X_r) P^\trans{}\\
 &= P \diag( X_1^{-1} A_1' X_1, \dots, X_r^{-1} A_r' X_r ) P^\trans{},
\end{align*}
which proves the first part of the claim.

The important piece of the claim is that $A_k = X_k^{-1} A_k' X_k$ has all nonzero elements, for generic $K' \in \kappa_\tensor{A}$. Since $\tensor{A}$, $V$, $Z$, and $P$ are fixed, $X_k$ and $X_k^{-1}$ are matrices of constants, while the coordinates of $A_k'$ are considered variables, i.e., $A_k \in \FF[ a_{k,ij} \mid 1\le i,j\le d]$ for $k=1,\ldots,r$. Hence,
\[
 0 = (A_k)_{ij} = (X_k^{-1} A_k' X_k)_{ij} = \sum_{p=1}^d \sum_{q=1}^d y_{k,ip} x_{k,qj} a_{k,pq}
\]
is a linear equation in the variables $a_{k,ij}$, where $X_k=[x_{k,ij}]$ and $X^{-1}=[y_{k,ij}]$. Clearly, $A_k$ has an element equal to zero if and only if the single polynomial equation $\prod_{i,j} (A_k)_{ij}$ vanishes.
Thus, the matrices $A_k$ with some nonzero elements \changed{are contained} in a Zariski closed set. It suffices to exhibit one $K' \in \kappa_\tensor{A}$ with all nonzero entries in all $A_k$'s to conclude that the foregoing closed set is a strict subset of $\kappa_\tensor{A}$.
Let $A'_k = X_k (\Delta_k + \mathbf{1} \mathbf{1}^\trans{} - \Id) X_k^{-1}$, where $\Delta_k$ is as in \cref{eqn_good_traceless}. Note that $\mathbf{1} \mathbf{1}^\trans{} - \Id$ is the matrix of ones, except on the diagonal where it is zero. With this choice of $A_k'$, we see that $A_k = \Delta_k + \mathbf{1}\mathbf{1}^\trans{} - \Id$, which is traceless and has all its entries different from $0$. This concludes the proof.
\end{proof}

Note how the notation emphasizes that, evidently, one cannot take $K$ from \cref{lem_evd} and $K'$ from \cref{lem_permutation} equal to one another.

\Cref{lem_permutation} suggests a procedure for identifying $P$. Let $Z$ and $K'$ be as in the lemma, and let $C = Z^{-1} K' Z$. Then, we can determine a permutation $P'$ that block diagonalizes $C$ simply by inspecting the nonzero elements of $C$ and building the permutation greedily; see \cref{sec_sub_partitioning} for concrete details. The resulting permutation $Q$ is not guaranteed to be equal to $P$. Nevertheless, $P$ and $Q$ are related through the existence of a permutation $\pi$ so that if $P^\trans{} C P = \diag(A_1, \dots, A_r),$ then $Q^\trans{} C Q = \diag(A_{\pi_1},\dots,A_{\pi_r})$. This implies that applying $P^\trans{} Q$ on the right permutes the block columns according to $\pi$.
That is, using the notation from the proof of \cref{lem_permutation}, since $Z P = V \diag(X_1, \dots, X_r)$, we have
\[
Z Q
= V \diag(X_1, \dots, X_r) P^\trans{} Q 
= \begin{bmatrix} V_{\pi_1} X_{\pi_1} & \dots & V_{\pi_r} X_{\pi_r} \end{bmatrix}.
\]
Thus, $ZQ = \begin{bmatrix} Z_{\pi_1}' & \dots & Z_{\pi_r}' \end{bmatrix}$, where the $Z_i$'s are as in \cref{lem_evd}. This means that the same elementary tensors $\vect{v}_i^1 \wedge\dots\wedge \vect{v}_i^d$ are identified, up to scale, except in a different order. Consequently, the correct \emph{set} of elementary skew-symmetric tensors, up to scale, are identified by the partitioning induced from $Q$.

\subsection{The high-level algorithm}

We are now ready to combine the foregoing ingredients into a mathematical algorithm for \changed{exact} low-rank Grassmann decomposition \changed{of noise-free tensors}. This algorithm is presented as \cref{alg_main_computation}. Its correctness is established by the next result.

\begin{algorithm}[tb]
\caption{Mathematical algorithm for low-rank Grassmann decomposition}
\label{alg_main_computation}
\algrenewcommand\alglinenumber[1]{{\small S{\the\numexpr#1-1\relax}.}}
\begin{algorithmic}[1] 
\Require{Tensor $\tensor{T}\in\wedge^d \VS{W}$ is generic of Grassmann rank $r \le \frac{1}{d} \dim\VS{W}$.}
\State{Compute an orthonormal basis $U$ of the column span of $\tensor{T}_{(1)}$ and express $\tensor{T}$ in it: $\tensor{A}=(U^\herm,\dots,U^\herm)\cdot\tensor{T}$;}
\State{Compute the matrix $J$ of the map $\delta_\tensor{A} : \FF^{dr \times dr} \to \wedge^d \FF^{dr},\, \dot{X} \mapsto \sum_{k=1}^d \dot{X} \cdot_k \tensor{A}$;}
\State{Compute the kernel $\kappa_\tensor{A}$ of $J$;}
\State{Choose a generic matrix $K \in \kappa_\tensor{A}$ and compute an EVD $K = W \Lambda W^{-1}$;}
\State{Choose a generic matrix $K' \in \kappa_\tensor{A}$ and compute a permutation $P$ such that $P^\trans{} W^{-1} K' W P$ is a block diagonal matrix;}
\State{Partition $WP = \begin{bmatrix} V_1 & \dots & V_r \end{bmatrix}$ with $V_i := \begin{bmatrix} \vect{v}_{i}^1 &\dots& \vect{v}_i^d \end{bmatrix}$ and improve the $V_i$'s;}
\State{Solve the linear system $\sum_{i=1}^r x_i \vect{v}_i^1\wedge\dots\wedge \vect{v}_i^d = \tensor{A}$ for $x_1,\dots,x_r$;}
\State Compute the factor matrix $D = \begin{bmatrix} x_1 U V_1 & \dots & x_r U V_r \end{bmatrix}$;
\State \Return $D$.
\end{algorithmic}
\end{algorithm}

\begin{theorem}\label{thm_main_algorithm}
Let 
\(
  \tensor{T} \in \wedge^d \VS{W} \simeq \wedge^d \FF^m
\)
be a generic skew-symmetric tensor of Grassmann rank $r$ with $m \ge dr$ and $d \ge 3$. Then, \cref{alg_main_computation} computes a set of elementary skew-symmetric tensors decomposing $\tensor{T}$.
\end{theorem}
\begin{proof}
By \cref{lem_compression}, we can focus on concise tensor spaces. 
A higher-order singular value decomposition (HOSVD) \citep{dLdMV2000} will compute orthonormal bases for the concise tensor product subspace $\wedge^d \VS{V}$ \changed{containing} $\tensor{T}$. By \cref{lem_wedge_flattening}, the standard flattenings are equal up to sign, so the HOSVD can be computed as in S0; \changed{see also \cite{BK2017}}. By the definition of multilinear multiplication, applying $U \otimes\dots\otimes U$ to any elementary skew-symmetric tensor $\vect{v}_i^1\wedge\dots\wedge \vect{v}_i^d$ yields the same elementary tensor $(U\vect{v}_i^1)\wedge\dots\wedge(U\vect{v}_i^d)$ but embedded in the original ambient space $\wedge^d \VS{W}$. This proves the correctness of step S7 of the theorem.

The correctness of steps S1--S5 follows immediately from combining \cref{thm_main,lem_evd,lem_permutation}. These results also show that a set of elementary skew-symmetric tensors $\{ \vect{v}_i^1\wedge\dots\wedge \vect{v}_i^d \mid i \in [r] \}$ will be identified, so that the solution of the linear system from step S6 in the theorem yields the rank-$r$ Grassmann decomposition of $\tensor{A} = (U^\herm,\dots,U^\herm)\cdot \tensor{T}$.

Note that the linear system $\tensor{A} = \sum_{i=1}^r x_i \vect{v}_i^1 \wedge\dots\wedge \vect{v}_i^d = \sum_{i=1}^r x_i \tensor{A}_i$ has a unique solution. Indeed, it could have multiple solutions $x_i$ with the same set of elementary tensors only if these tensors are not linearly independent. However, if this were the case, say $\tensor{A}_1 = x_2' \tensor{A}_2 + \dots + x_r' \tensor{A}_r$, without loss of generality, then we could factorize
\(
 \tensor{A} = \sum_{i=2}^r (1 + x_i') \tensor{A}_i,
\)
which contradicts the assumption that $\tensor{T}$, and, hence, $\tensor{A}$, has Gr-rank equal to $r$.
\end{proof}

\section{An efficient numerical implementation} \label{sec_numerical_algorithm}

This section discusses a concrete realization of \cref{alg_main_computation} as a numerical method in coordinates, suitable for implementation in floating-point arithmetic. 
It was designed for \emph{decomposition} of a tensor that is mathematically of low Gr-rank,
not as an \emph{approximation} method \changed{for finding a nearby low Gr-rank tensor given an arbitrary input}. When the Grassmann decomposition model \cref{eqn_grassmann_decomp} holds only approximately, one may use \cref{alg_main_computation} as an initialization method for an optimization-based algorithm, \changed{as is commonly done for tensor rank and block term decomposition in state-of-the-art tensor packages \cite{Tensorlab}.} The numerical experiments in \cref{sec_numerical_experiments} will investigate insofar as the concrete numerical implementation presented in this section can withstand \changed{small} model violations.
 
\changed{In my implementation,} tensors in $\wedge^d \VS{W}$ are represented in coordinates with respect to a basis 
\begin{equation}\label{eqn_standard_basis}
 ( \vect{w}_{i_1} \wedge\dots\wedge \vect{w}_{i_d} \mid 1 \le i_1 < \dots < i_d \le m),
\end{equation}
where $(\vect{w}_1,\dots,\vect{w}_m)$ is a basis of $\VS{W}$, by \cref{lem_basis}.
In this way, a skew-symmetric tensor is compactly represented as a vector of length $\dim \wedge^d\VS{W} = \binom{m}{d}$. \changed{Another option would be to treat them as general tensors in $\VS{W}^{\otimes d}$. However, this requires $m^d$ coordinates, which is approximately $d!$ times more expensive than the foregoing minimal representation.}

In the following subsections, the main steps and the notation of \cref{alg_main_computation} \changed{is reprised}, and additional details \changed{are provided} on how they can be implemented efficiently. 
The asymptotic time complexity of the proposed numerical implementation of \cref{alg_main_computation} is summarized in \cref{tab_complexity}. These complexity estimates are obtained by retaining the highest-order terms in the individual complexity analyses presented in the next subsections.

\begin{table}
\caption{The asymptotic time complexities of each step of the numerical implementation of \cref{alg_main_computation} from \cref{sec_numerical_algorithm} for decomposing a generic Gr-rank $r$ skew-symmetric tensor $\tensor{A} \in \wedge^d \FF^m$ into elementary skew-symmetric tensors. The integer $n = dr$.}
\newcommand{\ccenter}[1]{\hfill#1\hfill}
\[
 \begin{array}{rrrrrrrrr}
 \toprule
  \ccenter{\textbf{S0}} & \ccenter{\textbf{S1}} & \ccenter{\textbf{S2}} & \ccenter{\textbf{S3}} & \ccenter{\textbf{S4}} & \ccenter{\textbf{S5}} & \ccenter{\textbf{S6}} & \ccenter{\textbf{S7}} & \ccenter{\textbf{total}} \\
  \midrule 
  \changed{m^2 \binom{m}{d-1}} & \changed{n^2 m^2 \binom{m}{d-2}} & n^6 & n^3 & n^3 & d^2 n^3 & \changed{dmn \binom{m}{d-1}} & mn^2 & \text{S0} + \text{S1} + \text{S2} + \text{S6} \\
  \bottomrule
 \end{array}
\]
\label{tab_complexity}
\end{table}

\subsection{S0: Computing \changed{a basis of the concise tensor space}} \label{sec_sub_hosvd}

Given a tensor $\tensor{T} \in \wedge^d \VS{W}$, we can identify the concise tensor space $\wedge^d \VS{V}$ \changed{that contains $\tensor{T}$} by computing the image of $\tensor{T}_{(1)}.$ Indeed, if $\tensor{T} \in \wedge^d \VS{V} \subset \wedge^d \VS{W}$, then the $1$-flattening satisfies
\[
\tensor{T}_{(1)} \in \VS{V} \otimes (\wedge^{d-1} \VS{V})^* \subset \VS{W} \otimes (\wedge^{d-1} \VS{W})^*.
\]
The image of $\tensor{T}_{(1)}$ coincides with $\VS{V}$, for otherwise $\wedge^{d} \VS{V}$ would not be concise.

\changed{The mathematical \cref{alg_main_computation} suggests to explicitly express $\tensor{T}$ in coordinates with respect to an orthonormal basis of its concise space $\wedge^d \VS{V}$.}
\changed{However, computing $\tensor{A}=(U^\herm,\dots,U^\herm)\cdot\tensor{T}$ with high speed and low memory consumption seems to be hard in practice.
A few natural strategies are as follows.
First, the standard algorithm for multilinear multiplication using flattenings, matrix multiplication, and circular shifts (see, e.g., \cite[Section 4.1.2]{BV2023}) performs very well in terms of computational throughput, but requires asymptotically more operations, and, more significantly, requires a representation of $\tensor{T}$ as an $m\times\dots\times m$ array, which consumes $d!$ times more memory than exploiting its representation with $\binom{m}{d}$ coordinates in the basis \cref{eqn_standard_basis}. Second, a relatively technical algorithm was described in \cite[Section 5.1]{Vannieuwenhoven2024b} that exploits the partially skew-symmetric structures that arise when specializing the aforementioned algorithm; while it theoretically has a better time complexity, its computational throughput was low due to the unfavorable memory access patterns of flattenings and inverse flattenings. Third, computing the $\binom{n}{d}$ entries of $\tensor{A}$, with $n=dr$, as 
\[
a_{i_1 \dots i_d} = \sum_{1 \le j_1 < \dots < j_d \le m} \sum_{\sigma\in\mathfrak{S}([d])} \sign(\sigma) \overline{u_{i_1 j_1} \cdots u_{i_d j_d}} \cdot t_{j_1 \dots j_d}
\]
requires no additional memory but does involve computing and summing over $\binom{m}{d}$ elements, yielding a complexity of at least $\binom{m}{d} \binom{n}{d}$ operations; this is usually much more than the $d n m^{d}$ cost of the first algorithm.}

\changed{A careful inspection of \cref{alg_main_computation} reveals that $\tensor{A} = (U^\herm, \dots, U^\herm)\cdot\tensor{T}$ is used only in lines S1 and S6. We will see in \cref{sec_sub_gram,sec_sub_system_solve} how these lines can be executed without access to $\tensor{A}$. By circumventing the explicit computation of $\tensor{A}$, an overall speedup factor of over $3\times$ was obtained relative to the algorithm in \cite[Section 5.1]{Vannieuwenhoven2024b} for the computation of a Gr-rank $10$ decomposition of a tensor in $\wedge^6 \FF^{65}$, one of the most challenging cases considered in this article.}

\subsubsection{Flattenings}\label{sec_sub_flattening}
The $1$-flattening of a skew-symmetric tensor is computed as is suggested implicitly by \cref{lem_wedge_flattening}: loop over all $\binom{m}{d}$ coordinates of $\tensor{T}$ and put each of them into the $d$ correct positions of $\tensor{T}_{(1)}$ with the appropriate sign. 

The time complexity is $\BigO{ d^2 \binom{m}{d} }$ operations. 

\subsubsection{Basis}
While one can choose any basis of the image of $\tensor{T}_{(1)}$, it is recommended to choose an orthonormal one. 
Such a basis can be computed in many ways, including from an SVD, pivoted QR decomposition, or randomized methods.

In a numerical context it rarely happens that $\tensor{T}$ lies exactly in a lower-dimensional skew-symmetric subspace $\wedge^d \VS{V}$, because of various sources of imprecision such as approximation, computational, measurement, and round-off errors. Hence, we should seek a concise space \changed{close to} $\tensor{T}$, i.e., a space $\wedge^d \VS{V}$ such that the residual of the orthogonal projection of $\tensor{T}$ onto $\wedge^d \VS{V}$ is sufficiently small. If $U \in \FF^{m\times n}$ \changed{contains} an orthonormal basis \changed{of $\VS{V}$ in its columns}, then $\tensor{T}_\star = (U U^\herm, \dots, U U^\herm) \cdot\tensor{T}$ is the orthogonal projection of $\tensor{T}$ onto $\wedge^d \VS{V} \subset \wedge^d \VS{W}$, so step S0 of \cref{alg_main_computation} also applies in this approximate sense. To determine a suitable skew-symmetric space \changed{close to} $\tensor{T}$, we can use the truncated SVD \changed{of $\tensor{T}_{(1)}$. This results in a} quasi-optimal \changed{approximation $\tensor{T}_\star$ of $\tensor{T}$, which was already remarked in \cite[Section 2.2]{BK2017}.}

If the Gr-rank of $\tensor{T}$ is known, then the truncation multilinear rank for T-HOSVD should be chosen equal to $(dr,\dots,dr)$.
However, if it is unknown, then, under the assumption that \cref{alg_main_computation} applies, the multilinear rank of $\tensor{T}$ should be of the form $(dr,\dots,dr)$. When using a numerical thresholding criterion based on the singular values of $\tensor{T}_{(1)}$, this should be taken into account. For example, we can truncate based on the geometric means of $d$ consecutive singular values, i.e., 
\(
\sigma_i' := \sqrt[d]{ \sigma_{d(i-1)+1} \cdots \sigma_{di} },
\)
choosing the numerical $\epsilon$-rank as the largest index $i$ such that $\sigma_i' \ge \epsilon \sigma_1'$. \changed{The Gr-rank of $\tensor{T}$ can thus be determined based on the (numerical) rank of $\tensor{T}_{(1)}$, which is $\tensor{T}$'s order $d$ multiplied with the Gr-rank $r$.}

In my implementation, a standard rank-$rd$ truncated SVD \changed{based on LAPACK's standard SVD implementation was chosen} to determine an orthonormal basis of the approximate image of \changed{the $m \times \binom{m}{d-1}$ matrix} $\tensor{T}_{(1)}$. The asymptotic time complexity to \changed{compute an orthonormal basis of the column span of $\tensor{T}_{(1)}$ via a truncated SVD} is $\BigO{m^2 \binom{m}{d-1}}$ operations.

\changed{Note that a truncated randomized SVD may further lower the computational complexity at the expense of a bit of accuracy and determinism \cite{HMT2011,MT2020}. However, due to the unfavorable wide shape of $\tensor{T}_{(1)}$, specialized, structured tensor sketches should be used to achieve computational and memory efficiency. Such approaches have been extensively studied in the literature for unstructured tensors, e.g., \cite{WTSA2015,BBK2018,MB2020,CJ2021,MS2021}, but only sparse results for symmetric tensors exist \cite{WTSA2015}. See Pearce and Martinsson \cite{PM2025} for a recent survey on matrix and tensor sketching methods.
Determining efficient randomized sketches for (skew-)symmetric tensors represented intrinsically with $\binom{m}{d}$ coordinates seems to be an open problem and may require more advanced data structures with modest excess memory requirements \cite{SLVdGK2014}.}

\subsection{S1: Representing the map \texorpdfstring{$\delta_\tensor{A}$}{dA}} \label{sec_sub_gram}
To compute the kernel $\kappa_\tensor{A}$ of $\delta_\tensor{A}$, the straightforward approach consists of \changed{computing $\tensor{A}$ and} building the $\binom{n}{d} \times n^2$ matrix $J$ that represents $\delta_\tensor{A}$ in coordinates.
We can use standard numerical linear algebra libraries to compute its kernel, for example by extracting it from a full SVD. While this approach is accurate, it is relatively slow because of its $\BigO{\binom{n}{d} n^4}$ time complexity.

An alternative approach consists of computing the $n^2 \times n^2$ Gram matrix $G = J^\herm J$ whose kernel mathematically coincides with the one of $J$. 
Then, we only need to compute the kernel of a $n^2 \times n^2$ matrix, which requires $\BigO{n^6}$ operations if a standard SVD is used. 
As is often the case with multilinear maps, $G$ can be computed efficiently without constructing the large intermediate matrix $J$. Such an algorithm \changed{is described} next.

The Gram matrix $G$ is a Hermitian matrix in $(\VS{V} \otimes \VS{V}^*) \otimes \overline{(\VS{V} \otimes \VS{V}^*)}^*$, where the overline denotes the complex conjugation isomorphism. Hence, after choosing an orthonormal basis $(\vect{e}_1,\dots,\vect{e}_n)$ of $\VS{V} \simeq \FF^n$, it has a natural indexing by tuples $(i,j),(i',j')$, which \changed{we} abbreviate to $ij,i'j'$.
The entries of $G$ are by definition $G_{ij,i'j'} := \langle \delta_\tensor{A}(E_{ij}), \delta_\tensor{A}(E_{i'j'}) \rangle_F$, where $E_{ij} = \vect{e}_i \vect{e}_j^\herm$. Then, we compute
\begin{align}
G_{ij,i'j'}
\nonumber&= \sum_{k=1}^d \sum_{\ell=1}^d \langle  E_{ij} \cdot_k \tensor{A}, E_{i',j'} \cdot_\ell \tensor{A} \rangle_F\\
\nonumber&= \sum_{k=1}^d \mathrm{tr}( E_{ij} \tensor{A}_{(k)}\tensor{A}_{(k)}^\herm E_{i'j'}^\herm ) + \sum_{1\le\ell\ne k\le d} \mathrm{tr}\left( (E_{ij} \otimes I_n ) \tensor{A}_{(k,\ell)} \tensor{A}_{(k,\ell)}^\herm (I_n \otimes E_{i'j'})^\herm \right)\\
\label{eqn_gram_final}&= \sum_{k=1}^d \mathrm{tr}( \tensor{A}_{(k)}\tensor{A}_{(k)}^\herm E_{j'i'} E_{ij} ) + \sum_{1\le\ell\ne k\le d}  \mathrm{tr}\left( \tensor{A}_{(k,\ell)} \tensor{A}_{(k,\ell)}^\herm (E_{ij} \otimes E_{j'i'}) \right),
\end{align}
where $I_n$ is the matrix of $\Id_\VS{V}$.
Note that $E_{ij} \otimes E_{j'i'}$ is the tensor product of linear maps, so
$E_{ij}\otimes E_{j'i'} = (\vect{e}_i \vect{e}_j^\herm)\otimes (\vect{e}_{j'}\vect{e}_{i'}^\herm) = (\vect{e}_i\otimes\vect{e}_{j'})(\vect{e}_j\otimes\vect{e}_{i'})^\herm$. 
Let $G^{k,\ell} = \tensor{A}_{(k,l)} \tensor{A}_{(k,l)}^\herm$. We see that $G^{k,\ell} = G^{1,2}$ because all $(k,\ell)$-flattenings of the skew-symmetric tensor $\tensor{A}$ are the same, up to a sign $\pm1$, by \cref{lem_wedge_flattening}. Since we are multiplying the matrix with its adjoint, the sign is squared and disappears. The same observation holds for the $k$-flattenings and their Gram matrix $H = \tensor{A}_{(k)} \tensor{A}_{(k)}^\herm = \tensor{A}_{(1)}\tensor{A}_{(1)}^\herm$.

\changed{As remarked in \cref{sec_sub_hosvd} we do not explicitly form $\tensor{A}$. Instead, since $\tensor{A}=(U^\herm,\dots,U^\herm)\cdot\tensor{T}$, we determine that
\begin{align*}
H = \tensor{A}_{(1)} \tensor{A}_{(1)}^\herm 
&= U^\herm \tensor{T}_{(1)} (\overline{U} \otimes\dots\otimes \overline{U}) (\overline{U} \otimes\dots\otimes \overline{U})^\herm \tensor{T}_{(1)}^\herm U \\
&= U^\herm \tensor{T}_{(1)} \overline{({U} U^\herm \otimes\dots\otimes {U} U^\herm) \tensor{T}_{(1)}^\trans{}} U.
\end{align*}
Now we observe that $U U^\herm$ is an orthogonal projection onto the column span of $U$. Since the latter contains an orthonormal basis of the vector space $\VS{V} \subset \VS{W}$ and $\tensor{T} \in \wedge^d \VS{V}$, the projection $(\Id_{\VS{W}}, U U^\herm, \dots, U U^\herm) \cdot \tensor{T} = \tensor{T}$. Consequently,
\[
H = U^\herm \tensor{T}_{(1)} \overline{\tensor{T}_{(1)}^\trans{}} U = (U^\herm \tensor{T}_{(1)}) (U^\herm \tensor{T}_{(1)})^\herm.
\]
We analogously find the following expression for $G^{1,2}$:
\[
G^{1,2} = \bigl( (U \otimes U)^\herm \tensor{T}_{(1,2)} \bigr) \bigl( (U \otimes U)^\herm \tensor{T}_{(1,2)} \bigr)^\herm.
\]}%

By exploiting the above observations and the fact $E_{j'i'}E_{ij} = \delta_{ii'} E_{j'j}$, where $\delta_{ii'}$ is the Kronecker delta,
we can further simplify \cref{eqn_gram_final} to
\begin{align}\label{eqn_gram_final_final}
G_{ij,i',j'}
&= d \delta_{ii'} \vect{e}_{j}^\herm H \vect{e}_{j'} + d(d-1)  (\vect{e}_j \otimes \vect{e}_{i'})^\herm G^{1,2} (\vect{e}_i\otimes \vect{e}_{j'}).
\end{align}
Let $\sigma$ be the bijection that acts like 
\begin{align*}
\sigma : (\VS{V} \otimes \VS{V}) \otimes \overline{(\VS{V} \otimes \VS{V})}^* &\to (\overline{\VS{V}}^* \otimes \VS{V})\otimes \overline{( \overline{\VS{V}}^* \otimes \VS{V})}^*,\\
\vect{v}_j \otimes \vect{v}_{i'} \otimes \vect{v}_i^\herm \otimes \vect{v}_{j'}^\herm &\mapsto \vect{v}_i^\herm \otimes \vect{v}_j \otimes \vect{v}_{i'} \otimes \vect{v}_{j'}^\herm.
\end{align*}
Note that this map consists of permuting the coordinates of $G^{1,2}$.

\changed{In conclusion}, \cref{eqn_gram_final_final} can be computed efficiently as in \cref{alg_gram}.

\begin{algorithm}[tb]
\caption{Efficient Gram matrix construction}
\label{alg_gram}
\algrenewcommand\alglinenumber[1]{{\small S{\the\numexpr#1-1\relax}.}}
\begin{algorithmic}[1]
\Require{Skew-symmetric tensor $\tensor{T}\in \wedge^d \VS{V} \subset \wedge^d \VS{W}$ \changed{and an orthonormal basis $U$ of $\VS{V}$.}}
\State{\changed{$M_{12} \gets (U\otimes U)^\herm \tensor{T}_{(1,2)}$}}
\State{$G \gets \sigma\bigl( d(d-1) M_{12} M_{12}^\herm \bigr)$}
\State{\changed{$M \gets U^\herm \tensor{T}_{(1)}$}}
\State{$H \gets d M M^\herm$}
\For{$i \gets 1, \dots, n$}
\State{$G_{i,:,i,:} \gets G_{i,:,i,:} + H$}
\EndFor{}
\State \Return $G$.
\end{algorithmic}
\end{algorithm}

The $(1,2)$-flattening taking $\wedge^d \VS{W}$ to $\VS{W}^{\otimes 2} \otimes \wedge^{d-2} \VS{W}$ can be computed similarly as in \cref{sec_sub_flattening}: loop over all $\binom{m}{d}$ coordinates of $\tensor{T}$ and put them with the correct sign in the $d^2$ possible positions (requiring $\BigO{d}$ operations to compute each linear index).
With this information, the asymptotic time complexity of \cref{alg_gram} is determined to be
\changed{\begin{multline*}
 \underbrace{d \binom{d}{2} \binom{m}{d}}_\text{flattening in S0} + 
 \underbrace{n^2 m^2 \binom{m}{d-2}}_\text{matmul in S0} + 
 \underbrace{n^4}_\text{$\sigma$ in S1} + 
 \underbrace{n^4 \binom{m}{d-2}}_\text{matmul in S1} +
 \underbrace{d^2 \binom{m}{d}}_\text{flattening in S2} + 
 \underbrace{n m \binom{m}{d-1}}_\text{matmul in S2} + 
 \underbrace{n^2 \binom{m}{d-1}}_\text{matmul in S3} + 
 \underbrace{n^3}_\text{S4-S6}\\
= \BigO{ n^2 m^2 \binom{m}{d-2} }
\end{multline*}}
operations, \changed{where ``matmul'' refers to a matrix multiplication.}

\subsection{S2: Computing the kernel} \label{sec_sub_kernel}
The kernel of a linear map can be computed with numerical linear algebra libraries. \changed{It can be} computed accurately with a full SVD of the \changed{$n^2 \times n^2$} Hermitian matrix $G$, \changed{after which} the right singular vectors $\vect{k}_i$ corresponding to the $q = r(d^2 -1)$ smallest singular values \changed{are extracted}. These vectors, obtained as a vector of $n^2$ coordinates, \changed{are elements of} $\VS{V} \otimes \overline{\VS{V}}^*$, so they can be reshaped into $n \times n$ matrices, resulting in a unitary basis $\tensor{K}=(K_1, \dots, K_q) \in (\FF^{n \times n})^{\times q} \simeq \FF^{n\times n\times q}$ of the kernel $\kappa_\tensor{A}$ of $\delta_\tensor{A}$.

\subsection{S3: Performing an EVD} \label{sec_sub_s2} \label{sec_sub_sub_sample_kernel}

For simplicity, I choose $K = K_q$, the element that numerically lies closest to the kernel of $G$.
This choice may not satisfy the genericity condition from \cref{lem_evd}, but it is simple, efficient, and leads empirically to good accuracy.

An EVD of $K = W \Lambda W^{-1}$ can be computed with standard numerical linear algebra libraries. These libraries will compute it over $\mathbb{C}$ if there are pairs of complex conjugate eigenvalues.
This means that in the case of a real input tensor $\tensor{T}$, \cref{alg_main_computation} could require computations over $\mathbb{C}$ from step S3 onward. While this of no consequence in theory, it is somewhat inconvenient in practice, among others because it increases the cost of field multiplications by a factor of at least $3$ (using Gau\ss's algorithm).
Inspecting the proofs of \cref{lem_evd,lem_permutation} carefully shows that in the case of a real $K \in \mathbb{R}^{n \times n}$ \cref{lem_permutation} also holds if we compute a \emph{real} similarity transformation $K = W B W^{-1}$, with real $W \in \mathrm{GL}(\mathbb{R}^n)$, to a block diagonal form $B\in\mathbb{R}^{n \times n}$ with $1\times 1$ and $2 \times 2$ matrices on the diagonal \citep[Corollary 3.4.1.10]{HJ2013}; the $1 \times 1$ blocks correspond to the real eigenvalues, while the $2 \times 2$ blocks correspond to a pair of complex conjugate eigenvalues.
It is therefore recommended reducing a real $K$ with a real similarity transformation to block diagonal form.

LAPACK does not implement real similarity transformations to a block diagonal form, so I simply compute an EVD and check afterward if pairs of complex conjugate eigenvalues are present. For each pair of complex conjugate eigenvalues, it suffices to replace the corresponding pair of complex conjugate eigenvectors $\vect{v}$ and $\overline{\vect{v}}$ by the real vectors $\frac{1}{\sqrt{2}} (\vect{v} + \overline{\vect{v}})$ and $\frac{\imath}{\sqrt{2}} (\vect{v} - \overline{\vect{v}})$, respectively.

The computational complexity is dominated by the cost of computing an EVD of an $n \times n$ matrix, which asymptotically requires $\BigO{n^3}$ operations.

\subsection{S4: Partitioning into elementary tensors} \label{sec_sub_partitioning}

We can uniformly sample a unit-norm matrix $K'$ from $\kappa_\tensor{A}$ by sampling a random Gaussian vector $\vect{k} \sim N(\vect{0}, I)$ in $\FF^q$ and setting $K' = \sum_{i=1}^{q} k_i K_i / \|\vect{k}\|$. With probability $1$, this $K'$ is generic in the sense of \cref{lem_evd}.

The permutation is determined by first computing 
\[
L 
= (W^{-1}, W^\trans{}, \vect{k}/ \|\vect{k}\|) \cdot \tensor{K}
= \changed{W^{-1} \bigl( (\vect{k}/ \|\vect{k}\|) \cdot_3 \tensor{K} \bigr) W}.
\]
Then, for increasing $i=1,\ldots,n$, we greedily build a permutation $\vect{p}$ by appending the indices $\vect{i}$ of the $d$ largest elements in the $i$th column of $L$ to $\vect{p}$, provided none of the indices in $\vect{i}$ already appear in $\vect{p}$. If one of them does, then we proceed with the next column $i+1$ without appending to $\vect{p}$.

If $L$ is (approximately) a permutation of a block diagonal matrix $B$, then the above process recovers a vector $\vect{p}$ representing the permutation $P: [d] \to [d], i \mapsto p_i$ with the property that $L = P B P^\trans{}$. This is how $P$ from S4 in \cref{alg_main_computation} is obtained in my implementation.

The asymptotic time complexity for S4 is proportional to
\[
 \underbrace{ n^2 q + 3 n^3 }_\text{compute $L$} \;+\; \underbrace{n^2 \lg n}_\text{build $\vect{p}$} = \BigO{n^3}
\]
operations, assuming an $\BigO{n \lg n}$ sorting algorithm and a tree data structure with amortized $\BigO{\lg n}$ lookup and insertion cost are used to build $\vect{p}$.

\subsection{S5: Eigenbasis refinement} \label{sec_sub_sub_refinement}

After running steps S3 and S4 of \cref{alg_main_computation}, we obtain an initial factor matrix $V = \begin{bmatrix} V_1 & \dots & V_r \end{bmatrix}$. A consequence of \cref{thm_main} is that each $\VS{V}_i := \mathrm{span}(V_i)$ is a \emph{$\tensor{K}$-invariant subspace} \citep[Definition 1.3.16]{HJ2013} of all the matrices in $\kappa_\tensor{A}$, i.e.,
\(
 \mathrm{span}( K_j V_i ) = \VS{V}_i,
\)
for all $j=1,\dots,q$. Mathematically, \cref{lem_evd} ensures that $\VS{V}_i$ can be extracted from an EVD of one generic element $K$ in $\kappa_\tensor{A}$.
Numerically, however, the accuracy of extracting an invariant subspace of $K$ is limited by Sun's condition number \citep[Section 4.2]{Sun1991}, which depends nontrivially on the separation gap between the invariant subspaces $\VS{V}_i$ and $\oplus_{j\ne i} \VS{V}_j$.
The condition number of computing a $\tensor{K}$-invariant subspace is, a priori, different from Sun's condition number for computing the corresponding invariant subspace of $K$. 

In light of the numerical (in)stability results in \citep{BBV2019,BNV2023} and the numerical experiments in \cref{sec_sub_impact_refinement}, \changed{we will} improve the estimates of each invariant subspace $\VS{V}_i$ independently by the \emph{$\tensor{K}$-subspace iteration} method from \citep[Algorithm 2]{SS2024}; it is recalled in \cref{alg_fixed_point_improv}. It can be viewed as an iterated version of \citep[Algorithm 1]{AACCPV2021} and as a natural extension of classic subspace iteration \citep{Saad2011}. The notation in step S2 means that the first $d$ columns of $U$ are copied into $Q$.  

\begin{algorithm}[tb]
\caption{$\tensor{K}$-subspace iteration (Seghouane and Saad \citep[Algorithm 2]{SS2024})}
\label{alg_fixed_point_improv}
\algrenewcommand\alglinenumber[1]{{\small S{\the\numexpr#1-1\relax}.}}
\begin{algorithmic}[1]
\Require{A tensor $\tensor{K} = (K_1, K_2, \dots K_q)\in \FF^{n\times n \times q}$.}
\Require{A matrix $Q \in \FF^{n\times d}$ spanning an approximate $\tensor{K}$-invariant subspace.}
\Require{The number of iterations $p \in \mathbb{N}.$}
\For{$i \gets 1, \dots, p$}
\State{$U, S, V \gets \mathrm{SVD}\bigl( \begin{bmatrix}K_1 Q & K_2 Q & \dots & K_q Q \end{bmatrix} \bigr)$}
\State{$Q \gets U[:, 1:d]$}
\EndFor{}
\State \Return $Q$.
\end{algorithmic}
\end{algorithm}

This then leads to a time complexity proportional to 
\[
 \underbrace{n^2}_\text{compute $WP$} \;\;+\;\; \underbrace{2 p q n^2 d}_\text{S1 of \cref{alg_fixed_point_improv}} \;\;+\;\; \underbrace{p n d}_\text{S2 of \cref{alg_fixed_point_improv}} = \mathcal{O}( d^2 n^3 ),
\]
operations in S5 in \cref{alg_main_computation}, assuming that $p$ can be treated as a constant. My implementation uses $p=10$ iterations to refine the eigenbasis.

\subsection{S6: Solving the linear system}\label{sec_sub_system_solve}
Solving the overdetermined linear system in S6 using a standard QR decomposition is not recommended, as the matrix whose columns are the skew-symmetric tensors $\vect{v}_i^1\wedge\dots\wedge\vect{v}_i^d$ has size $\binom{n}{d} \times r$, which is very costly both in terms of memory and time.

To circumvent most of the above dual bottleneck, \changed{we can observe that the required coefficients $x_i$ in step S6 can be obtained by evaluating the multilinear map represented by the original tensor $\tensor{T}$ on appropriate vectors. This is understood as follows.
Let $V_i \in \FF^{n \times d}$, $i=1,\dots,r$, be the refined invariant subspaces resulting from the previous step. By our assumptions, $V = \begin{bmatrix} V_1 & \dots & V_r \end{bmatrix} \in \FF^{n \times n}$ is an invertible matrix. Let 
\[
V^{-1} := \begin{bmatrix} \Gamma_1^\herm \\ \vdots \\ \Gamma_r^\herm \end{bmatrix} \text{ with } \Gamma_i^\herm \in\FF^{d\times n}, \quad\text{so that}\quad \Gamma_j^\herm V_i = \delta_{ij} I_d \;\text{ for all }\; 1 \le i,j \le r,
\]
where $\delta_{ij}$ is the Kronecker delta and $I_d$ the $d\times d$ identity matrix.
We can express
\[
\tensor{A} 
= \sum_{i=1}^r x_i \vect{v}_i^1 \wedge\dots\wedge \vect{v}_i^d
= \sum_{i=1}^r x_i (V_i \vect{e}_1) \wedge\dots\wedge (V_i \vect{e}_d),
\]
where $V_i \vect{e}_j$ selects the $j$th column $\vect{v}_i^j$ of $V_i$.
Then, by multilinearly multiplying the original tensor $\tensor{T}$ with $\Gamma_j^\herm U^\herm$ on all factors, we find for every $j=1,\dots,r$ that
\[
\bigl( \Gamma_j^\herm U^\herm, \dots, \Gamma_j^\herm U^\herm \bigr) \cdot \tensor{T}
= \sum_{i=1}^r x_i (\Gamma_j^\herm V_i \vect{e}_1) \wedge\dots\wedge (\Gamma_j^\herm V_i \vect{e}_d)
= x_j \vect{e}_1 \wedge\dots\wedge \vect{e}_d.
\]
Combining this with \cref{eqn_wedge_product}, we can conclude that 
\[
\bigl(\vect{e}_1^\herm \Gamma_j^\herm U^\herm, \dots, \vect{e}_d^\herm \Gamma_j^\herm U^\herm \bigr) \cdot \tensor{T} 
= \frac{x_j}{d!} (\vect{e}_1^\herm, \dots, \vect{e}_d^\herm) \cdot \sum_{\sigma\in\mathfrak{S}([d])} \sign(\sigma) \vect{e}_{\sigma_1} \otimes\dots\otimes \vect{e}_{\sigma_d}
= \frac{x_j}{d!};
\]
the final equality can also be understood as a special case of \cref{eqn_alt_map_det}. This shows each $x_j / d!$ can be computed by evaluating $\tensor{T}$, considered as a multilinear map, on the tuple of vectors $(U \Gamma_j \vect{e}_1, \dots, U \Gamma_j \vect{e}_d)$.}

For general tensors, a multilinear multiplication can be computed efficiently by a sequence of \changed{flattenings and matrix-vector} multiplications. If $\tensor{T} \in \VS{W}^{\otimes d}$ and \changed{$\vect{f}_1, \dots, \vect{f}_d \in \VS{W}$ are vectors,} then the multilinear \changed{multiplication
\(
 (\vect{f}_1^\herm, \dots, \vect{f}_d^\herm) \cdot \tensor{T}
\)}
is computed efficiently as in \cref{alg_multilinear_mult}. 

\begin{algorithm}[tb]
\caption{Efficient multilinear multiplication}
\label{alg_multilinear_mult}
\algrenewcommand\alglinenumber[1]{{\small S{\the\numexpr#1-1\relax}.}}
\begin{algorithmic}[1]
\Require{Tensor $\tensor{T}\in \VS{W}\otimes\dots\otimes \VS{W}$ and \changed{vectors $\vect{f}_1,\dots,\vect{f}_d \in \VS{W}$}.}
\State{$\tensor{A}^0\gets\tensor{T}$}
\For{$k \gets 1, \dots, d$}
\State{$M_k \gets \tensor{A}_{(1)}^{k-1}$}
\State{$\tensor{A}_{(\emptyset; [d-k])}^k \gets \vect{f}_k^\herm M_k$}
\EndFor{}
\State \Return $\tensor{A}^d$.
\end{algorithmic}
\end{algorithm}

For skew-symmetric tensors, we can apply \cref{alg_multilinear_mult} as well. \changed{Observe that $\tensor{A}^k \in \wedge^{d-k} \VS{W}$ because in the $k$th iteration we multiply $\tensor{A}^{k-1} \in \wedge^{d-k+1}\VS{W}$ with $\vect{f}_k^\herm : \VS{W} \to \FF$ in the first factor by applying $\vect{f}_k^\herm$ to $\tensor{A}^{k-1}_{(1)} \in \VS{W} \otimes (\wedge^{d-k} \VS{W})^*$.
This results in an element of $\FF \otimes (\wedge^{d-k}\VS{W})^*$.  This is a row vector whose entries represent $\tensor{A}^{k}$ as an element of $\wedge^{d-k}\VS{W}$. The $1$-flattening can be computed as in \cref{sec_sub_flattening}, and the inverse of the $(\emptyset,[d])$-flattening consists of reinterpreting the row vector as a column vector, which requires no operations in practice.}
\changed{Exploiting this} in the implementation of \cref{alg_multilinear_mult} \changed{for skew-symmetric tensors}, yields an asymptotic time complexity of
\begin{equation}\label{eqn_complexity_mm}
C_\text{MM}^{m} =
\sum_{k=1}^{d} \Biggl(\underbrace{(d-k)^2 \binom{m}{d-k}}_{\text{$1$-flattening in S2}} + \underbrace{m \binom{m}{d-k}}_{\text{matmul in S3}} + \underbrace{0}_{\text{inverse flattening}}\Biggr)
= \BigO{ d (m + d^2) \binom{m}{d-1} }.
\end{equation}

\changed{In conclusion, the coefficients $x_j = d! \cdot ( (U \Gamma_j \vect{e}_1)^\herm, \dots, (U \Gamma_j \vect{e}_d)^\herm)\cdot \tensor{T}$ are computed by $r$ scaled multilinear multiplications via \cref{alg_multilinear_mult}. Therefore, the asymptotic time complexity of S6 is
\[
 \underbrace{ n^3 }_\text{compute $V^{-1}$} \;+\; \underbrace{r \cdot C_\text{MM}^{m}}_\text{multilinear multiplications} = \BigO{ n (m + d^2) \binom{m}{d-1} }
\]
operations.}

\subsection{S7: Computing the factor matrix} \label{sec_sub_factor_matrix}
This step can be implemented using numerical linear algebra libraries at an asymptotic cost of $\BigO{mndr}$ operations.

\section{Numerical experiments} \label{sec_numerical_experiments}

In this section, numerical experiments \changed{are presented} with my Julia \changed{v1.12.4} implementation of \cref{alg_main_computation} \changed{following the implementation choices from} \cref{sec_numerical_algorithm}. The only performance-critical external library used for the main implementation was \verb|LinearAlgebra.jl|, which relies on Julia's \verb|libopenblas64| and is configured to use $8$ threads. No explicit multithreading or parallelism is used elsewhere; most of the computationally demanding steps rely on the OpenBLAS implementation.
The implementation, including all code necessary to perform the experiments and generate the figures below can be found at \url{https://gitlab.kuleuven.be/u0072863/grassmann-decomposition}.

All the experiments \changed{were applied to} synthetic random tensors in $\wedge^d \FF^m$. A ``noiseless random Gr-rank-$r$ tensor'' is generated as follows. A random decomposition factor matrix $V\in\FF^{m \times dr}$ is sampled from the Gaussian ensemble, meaning $v_{ij}$ is sampled independently from the standard normal distribution $N(0,1)$ for all entries. For each $i=1,\ldots,r$, the elementary factor matrices $V_i$ are then normalized so that each column of $V_i$ has the same norm. The corresponding skew-symmetric tensor is generated by computing the wedge products of columns of the elementary factor matrices (with the algorithm from \cref{sec_sub_wedge}), and then summing all of these elementary skew-symmetric tensors. The tensor is then normalized so that its representation has unit Euclidean norm. Noise of level $\sigma$ can be added to the Gr-rank-$r$ tensor $\tensor{A}$, by sampling a Gaussian vector in $\FF^{\binom{m}{d}}$ in which each entry is sampled independently from $N(0,1)$. The vector is normalized to unit length, multiplied with $\sigma$, and added to $\tensor{A}$.

The performance measures that are used in the experiments are standard: a \textit{relative backward error}, a \textit{relative forward error}, and the wall clock \textit{execution time}. Let $\tensor{A} = \sum_{i=1}^r \vect{w}_i^1 \wedge\dots\wedge \vect{w}_i^d$ be the true Grassmann tensor and $\widehat{\tensor{A}} = \sum_{i=1}^r \widehat{\vect{w}}_i^1 \wedge\dots\wedge \widehat{\vect{w}}_i^d$ be its approximation computed by the proposed numerical algorithm. Then, the relative backward and forward errors are, respectively,
\[
\epsilon_{\text{b}} =
\frac{\| \widehat{\tensor{A}} - \tensor{A} \|_F}{\| \tensor{A} \|_F}
\quad\text{and}\quad 
\epsilon_{\text{f}} = \max_{i=1,\dots,r} \mathrm{dist}_{\mathrm{Gr}(d,\FF^m)}(\vect{w}_i^1 \wedge\dots\wedge \vect{w}_i^d, \widehat{\vect{w}}_i^1 \wedge\dots\wedge \widehat{\vect{w}}_i^d),
\]
where $\mathrm{dist}_{\mathrm{Gr}(d,\FF^m)}$ is the \emph{chordal distance} on the Grassmannian $\mathrm{Gr}(d,\FF^m)$:
\[
 \mathrm{dist}_{\mathrm{Gr}(d,\FF^m)} (U, \widehat{U}) = \frac{1}{\sqrt{2}} \| U U^\herm - \widehat{U} \widehat{U}^\herm \|_F
\]
if $U, \widehat{U} \in \FF^{m \times d}$ are matrices with orthonormal columns (in the Frobenius inner product) whose column spans represent the subspaces between which the distance is measured. Note that the order of the summands in a Grassmann decomposition is ambiguous. To determine the (hopefully) correct matching, the orthogonal projection of the elementary Grassmann tensors $\widehat{\tensor{A}}_i = \widehat{\vect{w}}_j^1\wedge\dots\wedge\widehat{\vect{w}}_j^d$ onto the $\tensor{A}_i=\vect{w}_i^1\wedge\dots\wedge\vect{w}_i^d$ are efficiently computed and organized into an $n \times n$ matrix $P$. In the case of a perfect decomposition, $P$ will be a permutation matrix that indicates how the $\widehat{\tensor{A}}_i$'s should be permuted to match up with the $\tensor{A}_i$'s. By continuity, $P$ will be close to a permutation matrix when $\tensor{A}\approx\widehat{\tensor{A}}$. We can then search the largest element in each column to determine a suitable permutation.

All experiments were executed on \verb|aerie|, a computer system running Xubuntu 24.04 LTS and featuring an AMD Ryzen 7 5800X3D (8 physical cores, 3.4GHz \changed{maximum} clock speed, 96 MB L3 cache) and $4\times32$ GB DDR4--3600 main memory.

\subsection{Impact of eigenbasis refinement} \label{sec_sub_impact_refinement}

The effect of the number of iterations $p$ of \cref{alg_fixed_point_improv} on the final backward error is investigated first. 
\changed{For each} $p = 0, 1, 2, 4, 8, 16, 32$, \changed{we independently sample} $100$ noiseless random Gr-rank-$15$ tensors in $\wedge^3 \mathbb{R}^{50}$. The Grassmann decompositions \changed{of these $700$ tensors} are then computed with \cref{alg_main_computation}. The resulting relative backward errors are visualized in \cref{fig_csi}.

\begin{figure}[tb]\centering
\includegraphics[width=.8\textwidth]{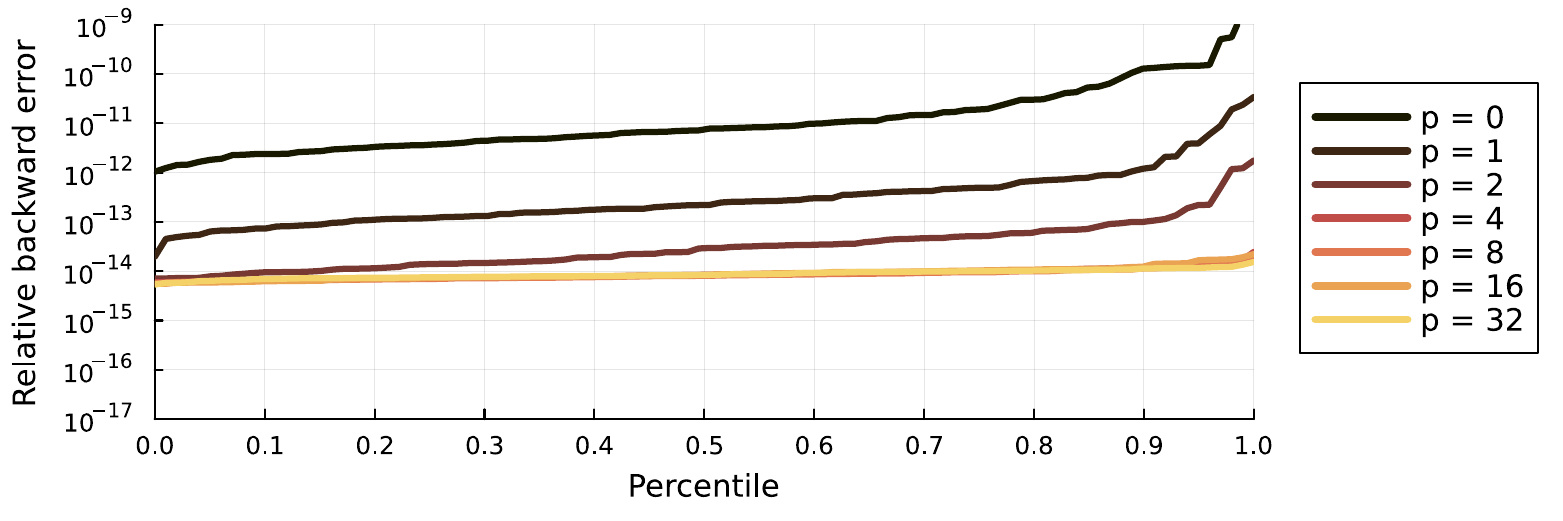}
\caption{The relative backward error of decomposing $100$ random Gr-rank-$15$ tensors in $\wedge^3 \mathbb{R}^{50}$ for a varying number of iterations $p$ in \cref{alg_fixed_point_improv}.}
\label{fig_csi}
\Description{A plot showing that the relative backward error of decomposing $100$ random Gr-rank-$15$ tensors in $\wedge^3 \mathbb{R}^{50}$ is substantially improved by using at least $p=4$ iterations in \cref{alg_fixed_point_improv}. Using this many iterations, the error reduces from about $10^{-11}$ at $p=0$ to about $10^{-14}$ at $p \ge 4$. Not much benefit is observed from increasing $p$ beyond $4$, as the corresponding curves are visually almost indistinguishable.}
\end{figure}

A dramatic improvement is observed from $p=0$ to $p=4$ of about $3$ orders of magnitude. Increasing $p$ further does not appear to offer any benefit in this experiment. Based among others on this experiment, $p=10$ was chosen as the default value of the number of $\tensor{K}$-subspace iterations in \cref{alg_fixed_point_improv}, offering a good trade-off between additional computational cost and accuracy.

\subsection{Computation time breakdown}

Next, we verify empirically insofar as the complexity estimates in \cref{tab_complexity} are representative of true performance. To validate this, one noiseless random Gr-rank-$r$ decomposition in $\wedge^d \mathbb{R}^{65}$ \changed{is generated}. The experiment aims to highlight and isolate the impact of increasing the order $d$ of the tensor, as the relative importance of the various steps in \cref{alg_main_computation} depends crucially on $d$ and the fraction $\frac{m}{n}$. For this reason, $n$ is chosen as the least common multiple of $3,4,5,$ and $6$, i.e., $n=60$, so that the fraction $\frac{m}{n} = \frac{65}{60}$ is constant and the Gr-rank $r = n / d$ is an integer for $d=3,4,5,6$.
The resulting breakdown of the execution time is shown in \cref{fig_fraction}.

\begin{figure}[tb]\centering
\includegraphics[width=.8\textwidth]{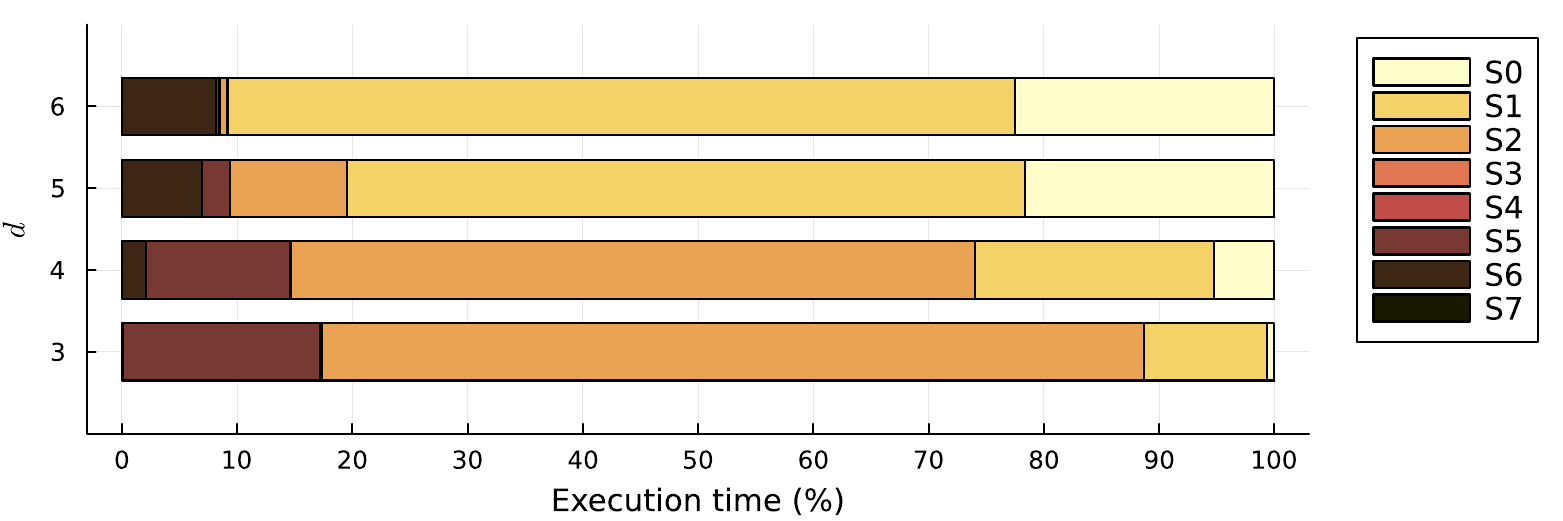}
\caption{Relative breakdown of the execution times of the steps in \cref{alg_main_computation} for a random Grassmann tensor in $\wedge^d \mathbb{R}^{65}$ of rank $r = 60/d$ for $d=3,4,5,6$. The absolute total execution times were $4.1$s, $7.6$s, $60.5$s, and $1208.3$s for, respectively, $d=3,4,5,$ and $6$.}
\label{fig_fraction}
\Description{The figure shows the following. For $d=6$, about $65\%$ of the time is spent on S1, $25\%$ on S0, and $10\%$ on S6. For $d=5$, about $55\%$ of the time is spent on S1, $25\%$ on S0, $10\%$ on S2, and $5\%$ on S6. For $d=4$, about $60\%$ of the time is spent on S2, $20\%$ on S1, and $10\%$ on S6. For $d=3$, about $70\%$ of the time is spent on S2, $20\%$ on S6, and $10\%$ on S1.}
\end{figure}

We observe in \cref{fig_fraction} that for all the orders $d$, steps S3 (EVD), S4 (block diagonalization), and S7 (extracting the factor matrix) virtually take up no time relative to the total execution time. This aligns well with the theoretical complexities in \cref{tab_complexity}, as these steps are of order $n^3$, while the others are at least of order $n^4$. 

Another observation that aligns well with \cref{tab_complexity} is the shrinking portion of the kernel computation in S2. For $d=3$, the $n^6$ complexity will usually be the dominant cost for \cref{alg_main_computation}. However, as this cost is independent of $d$, it will be quickly overtaken by S0, S1, and S6, whose complexity grows exponentially in $d$. Indeed, the kernel computation takes up a relatively insignificant amount of time as $d \ge 5$, while for $d \le 4$, it is empirically the dominant cost of running \cref{alg_main_computation}. 

\Cref{tab_complexity} suggests that the complexity of the eigenbasis refinement in S5 should be quite strongly dominated by the complexity of steps S0 and S1. However, in \cref{fig_fraction} we observe that S5 takes up a visible fraction of the execution time for $d \le 5$. This is attributed to the relatively large $10 d^2$ (the $10$ originates from the number of iterations for \cref{alg_fixed_point_improv}) in front of the $n^3$, whereas for S0 and S5 the $m^{d+1}$ complexity is scaled by the moderating coefficient $\frac{d}{(d-1)!}$.

Finally, the most computationally significant steps for large $d$ according to \cref{tab_complexity} are S0 and S1. Step S6 can also be significant if $m \approx n$, such as in this experiment, though it too will become relatively unimportant as $d$ keeps growing. \Cref{fig_fraction} empirically confirms the significance of S0, S1, and S6. The relative fraction of S0 and S1 do not appear to be in line with the theoretical prediction. I attribute this to the difference in the operations that underlie the leading complexity terms. In the case of S0, the leading term originates from an SVD, while for S1 it originates from a matrix multiplication. 
The difference in performance is then largely explained by (i) the lower constant in front of the time complexity of matrix multiplication, (ii) the higher attainable peak throughput for matrix multiplication, and (iii) the better parallel efficiency of matrix multiplication on shared-memory systems.

\subsection{Performance on random model tensors}

The next experiment aims to show the overall performance of the proposed \cref{alg_main_computation} on the three performance measures. I generate one noiseless random Grassmann tensor of rank $r$ in $\wedge^d \mathbb{R}^{dr}$ for all $3 \le d \le 10$ and $1 \le r \le 33$, subject to the constraint that $(dr)^2 \le 10^4$ and \changed{$2r \le \sqrt[d]{75 \cdot 10^6}$.} The numerical implementation of \cref{alg_main_computation} from \cref{sec_numerical_algorithm} is then used to decompose these tensors. The results are shown in \cref{fig_rank_and_degree}. 

\begin{figure}[tb]\centering 
\includegraphics[width=.8\textwidth]{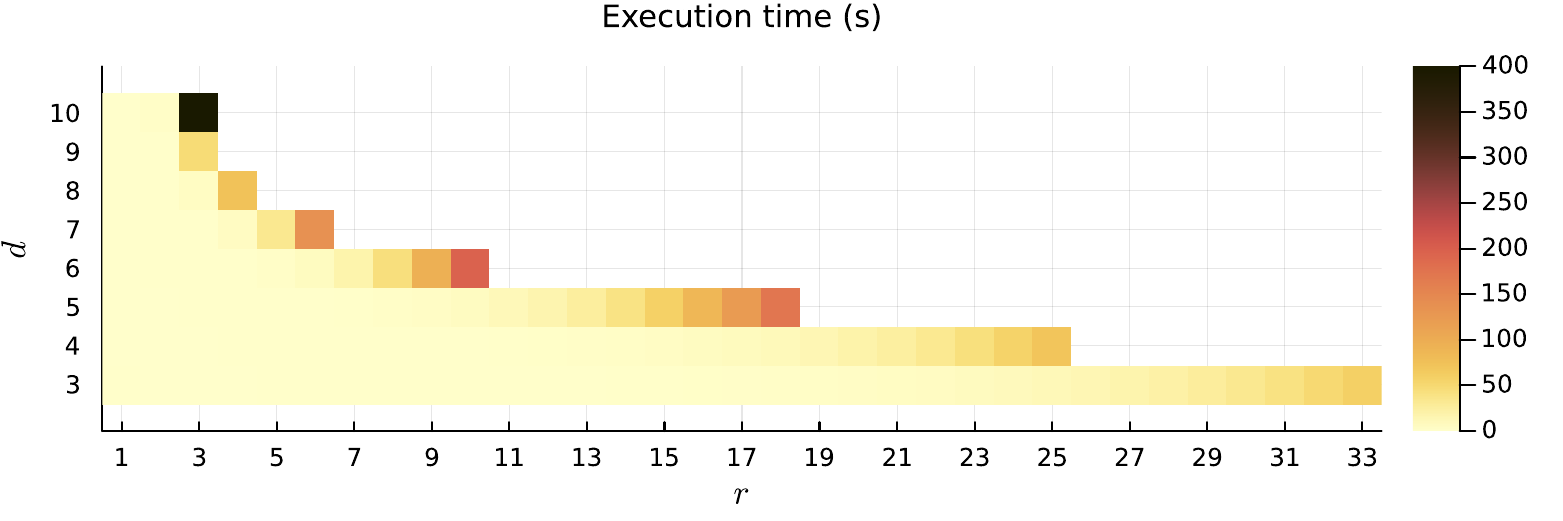}
\includegraphics[width=.8\textwidth]{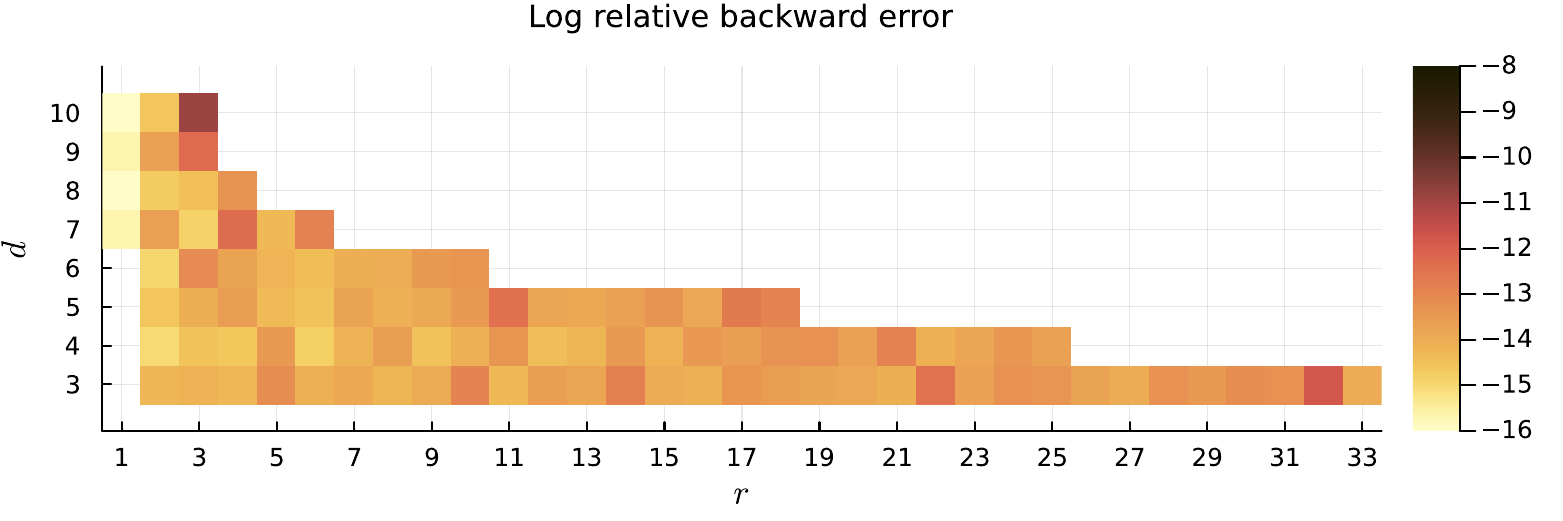}
\includegraphics[width=.8\textwidth]{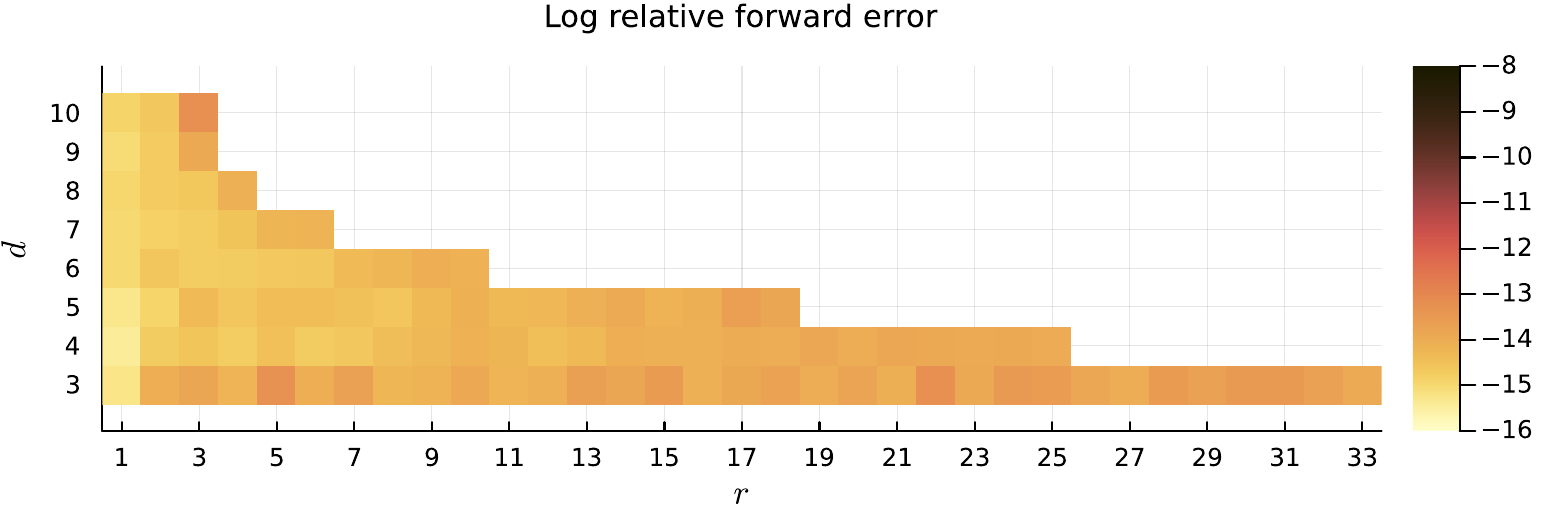}
\caption{The execution time in seconds and the base-$10$ logarithm of the relative backward and forward errors for decomposing noiseless random Gr-rank-$r$ tensors in $\wedge^d \FF^{dr}$ for feasible combinations of $3 \le d \le 10$ and $1 \le r \le 33$.}
\label{fig_rank_and_degree}
\Description{Complicated graphic showing small relative errors (backward and forward) in the range of $10^{-12}$ to $10^{-16}$.}
\end{figure}

The first panel in \cref{fig_rank_and_degree} shows the total execution time for decomposing the tensor. No particular observations stand out above and beyond what we knew theoretically from the complexity analysis in \cref{tab_complexity}. It is nonetheless interesting to see the absolute numbers and to observe the very competitive timings for $d=3$ up to $r \le 33$ (i.e., $n=99$), requiring less than $1$ minute. Note that by attempting to treat some of the higher-order tensors as general tensors, i.e., disregarding the skew-symmetric structure, we would not be able to compute their Grassmann decomposition. For example, a Gr-rank-$5$ tensor in $\wedge^7 \mathbb{R}^{35}$ requires $514.7$GB of storage as a general tensor, versus only $53.8$MB as a skew-symmetric tensor---a difference of almost four orders of magnitude.

The second panel of \cref{fig_rank_and_degree} illustrates the relative backward error. For all tested problems, the obtained Grassmann decomposition was very close (all \changed{$\epsilon_\text{b} \le 2\cdot10^{-12}$ but one}) to the original tensor in relative error. Note the seemingly missing values for $3 \le d \le 6$ for $r=1$. The reason is that the relative backward errors are exactly equal to $0$ in these cases, so their base-$10$ logarithm is $-\infty$. Note that for $r=1$, \changed{we do not need to} execute \cref{alg_main_computation} completely, as we can stop after S0 because of \cref{prop_rank1}.

The third panel of \cref{fig_rank_and_degree} shows the relative forward error. 
We observe visually that there is little difference between the backward and forward error. This suggests that the (projective) Grassmann decomposition problem seems to be well-conditioned for random Grassmann tensors when measuring errors in the space of tensors with the Frobenius norm and errors in the (projective) output space $\mathrm{Gr}(d,\mathbb{R}^n) \times\dots\times \mathrm{Gr}(d,\mathbb{R}^n)$ as the $\infty$-norm of the chordal distances in the respective Grassmannians. The condition number of Grassmann decomposition (with respect to the product norm in the codomain) can be determined by applying the techniques from \citep{BV2018}. Investigating this was out of the scope of the present work.

\subsection{Performance in the noisy regime}\label{sec_sub_noisy}

The final experiment investigates insofar as \cref{alg_main_computation}, which was designed as a decomposition algorithm for skew-symmetric tensors admitting an exact Grassmann decomposition, can cope with model violations. That is, how robust is the algorithm against arbitrary perturbations of an exact, true Grassmann decomposition. For this, $100$ noisy complex Gr-rank-$10$ decompositions in $\wedge^d \mathbb{C}^{50}$ \changed{are generated}, for each $d=3,4,5$, and for each of the noise levels $\sigma = 10^{-17}, 10^{-16.5}, 10^{-16}, \dots, 10^0$. The underlying true Grassmann decomposition is different in each random sample, so all data points are completely independent of one another. The maximum of the resulting relative forward errors, as compared to the true Grassmann decomposition of the noiseless tensor, is shown in \cref{fig_perturbation}.

\begin{figure}[tb]\centering
\includegraphics[width=.8\textwidth]{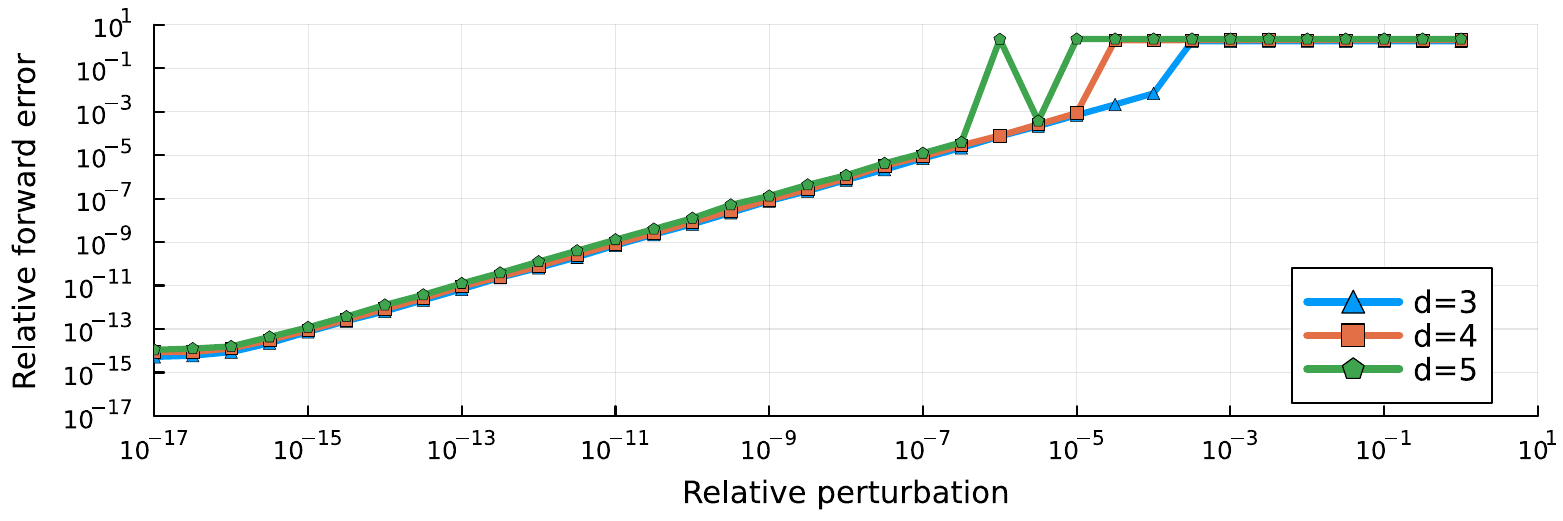}
\caption{The maximum relative forward error under relative perturbations for $100$ random noisy Gr-rank-$10$ tensors in $\wedge^d \mathbb{C}^{50}$ for $d=3,4,5$.}
\label{fig_perturbation}
\Description{Piecewise continuous curves are visible for $d=3,4,5$: for a relative perturbation between $10^{-17}$ and $10^{-16}$, the curves are horizontal at a relative forward error of about $10^{-14}$, for a relative perturbation between $10^{-16}$ to about $10^{-6}$ the curves rise linearly with slope $1$ in this log-log plot, indicating that the relative forward error grows proportionally to the relative perturbation, and finally from a relative perturbation higher than $10^{-6}$, the maximum forward error jumps up to $10$, indicating failure.}
\end{figure}

Observe in \cref{fig_perturbation} that from $\sigma = 10^{-16}$ to about $10^{-6}$ the relative forward error is equal to about $100 \sigma$, indicating a solid robustness to white Gaussian noise added to a true low-rank Grassmann tensor. For $\sigma \le 10^{-16}$, the relative forward error plateaus out. This is expected because only double-precision floating-point arithmetic was used and the noise drowns in the signal. For large noise levels, i.e., $\sigma \ge 10^{-6}$ for $d=5$ and $\sigma \ge 10^{-5}$ for $d \le 4$, the relative forward error suddenly jumps to $\approx 10^0$. This can be attributed partly to the failure of the matching algorithm in the computation of the relative forward error, and partly to the failure of \cref{alg_main_computation} in this high-error regime. 

\section{Conclusions}\label{sec_conclusions}
This article proposed \cref{alg_main_computation}, the first efficient numerical algorithm for real and complex Grassmann decomposition of generic order-$d$ skew-symmetric tensors in $\wedge^d \FF^m$ admitting such a decomposition up to small numerical perturbations and of rank $r \le \frac{m}{d}$. The technique is based on the framework of Brooksbank, Kassabov, and Wilson \citep{BKW2024}, which relies on the extraction of tensor decompositions from linearly-computable invariants for the tensor isomorphism problem. 
An efficient Julia implementation of the numerical algorithm from \cref{sec_numerical_algorithm} for both real and complex skew-symmetric tensors, represented intrinsically in $\wedge^d \FF^m \simeq \FF^{\binom{m}{d}}$ was developed. Numerical experiments support the claims about efficiency and accuracy in the case of random low-rank Grassmann decompositions.

It is an open question to what extent \cref{alg_main_computation} can be used as an effective initialization for optimization-based techniques for Grassmann decomposition in the high-noise regime, in light of the results in \cref{sec_sub_noisy}.
Another avenue for further study concerns the numerical stability of the proposed algorithm, in particular in the setting of ill-conditioned Grassmann decomposition problems. Brooksbank, Kassabov, and Wilson \citep{BKW2024} also proposed other chisels for sparsification, which might also be used to compute low-Gr-rank decompositions. Different chisels will have varying computational cost and are also anticipated to admit different numerical stability characteristics. 

\begin{acks}
This research was partly funded by the \grantsponsor{FWO}{Research Foundation---Flanders (FWO)}{https://www.fwo.be} with projects \grantnum{FWO}{G080822N} and \grantnum{FWO}{G011624N}.

A.~Taveira Blomenhofer is thanked for inviting me to the Centrum voor Wiskunde en Informatica (CWI) where he inspired me to think about Grassmann decompositions and elaborated on his recent identifiability results on February 19, 2024. 

The organizers of the \emph{Tensors: Algebra--Geometry--Applications (TAGA24)} conference at the Colorado State University Rocky Mountain Campus, Pingree Park, CO held from June 3 to 7, 2024 are thanked for this wonderful event, where I first learned about the research of P.~Brooksbank, M.~Kassabov, and J.~B.~Wilson.
I am indebted to H.~Abo for inquiring about the nondefectivity of secant varieties of the Grassmannian of $3$-dimensional subspaces on the bus to the Rocky Mountain Campus, hereby further setting my mind on Grassmann decompositions.
I thank P.~Brooksbank, M.~Kassabov, and J.~B.~Wilson for interesting discussions we had at TAGA24 on the sparsification of tensors through their chisel-based framework. In addition, I thank J.~B.~Wilson for his visit to Leuven on September 26, 2024, and the very insightful discussions we had with T.~Seynnaeve, D.~Taufer, and D.~Thorsteinsson.

T.~Seynnaeve and D.~Taufer are thanked for our discussions on this topic and for feedback they provided on an earlier version of this manuscript.

\changed{The three reviewers are thanked for their detailed comments that resulted in substantial improvements to this article: the expanded discussion of related decompositions and additional interpretations of skew-symmetric tensors in \cref{sec_introduction}; streamlining of the algebraic preliminaries in \cref{sec_preliminaries}; the expanded discussion of Brooksbank, Kassabov, and Wilson's chiseling framework applied to the tensor rank decomposition in \cref{sec_dleto}; a few additional clarifications in \cref{sec_algorithm}; and the simplified algorithms in \cref{sec_sub_hosvd,sec_sub_system_solve}.
The editor, Edgar Solomonik, is thanked for handling the process.}
\end{acks} 

\appendix 
\section{Elementary properties of the wedge product}
For convenience, a number of standard properties of the wedge product \cref{eqn_wedge_product} and skew-symmetric tensors \changed{are recalled. They} will be used freely throughout this paper; see \citep{Greub1978,Lang2002,Landsberg2012,Winitzki2023}.

The following well-known properties follow immediately from \cref{eqn_wedge_product} and the multilinearity of the tensor product \citep{Greub1978}.

\begin{lemma}\label{lem_elementary_props}
The following properties hold for all \changed{vectors} $\vect{v}_i, \vect{v} \in \VS{V}$ and \changed{scalars} $\alpha,\beta \in \FF$:
 \begin{enumerate}[1.]
\item Nilpotency: $\vect{v}_1 \wedge\dots\wedge \vect{v}_d = 0$ if and only if $\dim \langle \vect{v}_1,\ldots,\vect{v}_d \rangle < d$.
\item Anti-symmetry: $\vect{v}_{\sigma_1} \wedge\dots\wedge \vect{v}_{\sigma_d} = \sign(\sigma) \vect{v}_1 \wedge\dots\wedge \vect{v}_d$ for all $\sigma\in\mathfrak{S}([d])$.
\item Multilinearity: $(\alpha \vect{v}_1 + \beta \vect{v}) \wedge\vect{v}_2\wedge\dots\wedge \vect{v}_d = \alpha \vect{v}_1 \wedge\dots\wedge \vect{v}_d + \beta \vect{v} \wedge\vect{v}_2\wedge\dots\wedge \vect{v}_d$.
\end{enumerate}
\end{lemma}

Note that the linearity in the first factor (property 3), extends to multilinearity, i.e., linearity in any given factor, by exploiting property 2.

By definition, elementary skew-symmetric tensors \changed{are elements of} the tensor product $\VS{V}^{\otimes d}$. The linear space they span is denoted by $\wedge^d \VS{V} \subset \VS{V}^{\otimes d}$ and called the space of skew-symmetric tensors. The following result is standard; see, e.g., \citep[Chapter XIX]{Lang2002}.

\begin{lemma}\label{lem_basis}
If $\{ \vect{v}_1, \ldots, \vect{v}_n \}$ is a basis of $\VS{V}$, then $\{ \vect{v}_{i_1}\wedge\dots\wedge\vect{v}_{i_d} \}_{1 \le i_1 < \dots < i_d \le n}$ is a basis of $\wedge^d \VS{V}$.
Consequently, $\dim \wedge^d \VS{V} = \binom{n}{d}$.
\end{lemma}

If we equip $\VS{V}$ with an inner product $\langle\cdot,\cdot\rangle$, then there is an induced inner product in $\VS{V}^{\otimes d}$ \citep{Hackbusch2019}:
\[
 \langle\cdot,\cdot\rangle : \VS{V}^{\otimes d} \times \VS{V}^{\otimes d} \to \FF, \quad (\vect{v}_1\otimes\dots\otimes\vect{v}_d, \vect{v}_1'\otimes\dots\otimes\vect{v}_d') \mapsto \prod_{i=1}^{d} \langle \vect{v}_i,\vect{v}_i' \rangle.
\]

Multilinear multiplication interacts with skew-symmetric tensors as follows.

\begin{lemma}[Multilinear multiplication]\label{lem_mult_skew}
Let $A : \VS{V} \to \VS{W}$ be a linear map. Then,
\[
A\otimes\cdots\otimes A : \wedge^d \VS{V} \to \wedge^d \VS{W}, \quad \vect{v}_1\wedge\dots\wedge \vect{v}_d \mapsto (A \vect{v}_1) \wedge\dots\wedge (A\vect{v}_d).
\]
\end{lemma}
\begin{proof}
This follows immediately from the definition of the wedge product \cref{eqn_wedge_product}, linearity, and the definition of multilinear multiplication.
\end{proof}

Note the precise claim made in the previous lemma: the regular tensor product $A\otimes\cdots\otimes A$ when restricted to the subspace of skew-symmetric tensors maps into a space of skew-symmetric tensors. One could alternatively look at the natural action of $A$ on a skew-symmetric tensor, which would be defined exactly as in the lemma.

The next two well-known facts relate the wedge product to determinants.

\begin{lemma}\label{cor_det_transform}
Let $\VS{V} \subset \VS{W}$ be an $n$-dimensional subspace of $\VS{W}$. Let $\vect{v}_1, \ldots, \vect{v}_n$ and $\vect{v}_1', \ldots, \vect{v}_n'$ be bases of $\VS{V}$. Then,
\[
 \vect{v}_1 \wedge\dots\wedge \vect{v}_n = \det(X) \vect{v}_1' \wedge\dots\wedge \vect{v}_n',
\]
where $X\in\FF^{n\times n}$ is such that $V = V' X$ with $V=[\vect{v}_i]$ and $V'=[\vect{v}_i']$.
\end{lemma}

\begin{lemma}
 Choose a basis $\vect{v}_1,\dots,\vect{v}_n$ of $\VS{V}$. Then, \(\vect{v}_1 \wedge\dots\wedge \vect{v}_n = \det(V) \vect{e}_1 \wedge\dots\wedge \vect{e}_n,\)
 where $V$ is the matrix formed by placing the $\vect{v}_i$'s as columns and $\vect{e}_1,\dots,\vect{e}_n$ is the standard basis of $\FF^n$.
\end{lemma}

\section{Computing wedge products} \label{sec_sub_wedge}

The most naive way of computing the wedge product of $d$ vectors in a $m$-dimensional vector space consists of applying \cref{eqn_wedge_product}. Implemented as stated, this leads to a nasty complexity of $\Var{O}( d! m^d )$.

To circumvent this enormous complexity, \changed{we can} use the splitting suggested by \cref{lem_wedge_flattening} recursively. That is, in the notation of \cref{lem_wedge_flattening}, first recursively compute the $2\binom{d}{k}$ wedge products $\vect{v}_{\eta_1}\wedge\dots\wedge\vect{v}_{\eta_k}$ and $\vect{v}_{\theta_1}\wedge\dots\wedge\vect{v}_{\theta_{d-k}}$. They can be placed \changed{respectively} as the columns of two matrices $A$ and $B$, which are then multiplied to yield the $(\sigma; \rho)$-flattening $M=AB^\trans{}$ of $\vect{v}_1\wedge\dots\wedge\vect{v}_d$. The relevant coordinates can then be extracted from $M$.

If we denote the cost of computing the wedge product of $k$ vectors by $C_{\wedge}^{m,k}$, then the foregoing algorithm entails an upper bound of
\begin{equation}\label{eqn_recursive_bound}
C_{\wedge}^{m,d} \le \underbrace{\binom{d}{k} ( C_{\wedge}^{m,k} + C_{\wedge}^{m,d-k})}_\text{recursive computation} + \underbrace{2 \binom{m}{k} \binom{d}{k} \binom{m}{d-k}}_\text{matrix multiplication} + \underbrace{\binom{m}{d}}_\text{inverse flattening}
\end{equation}
elementary operations. The base case is $C_{\wedge}^{m,1} = m$. The complexity depends on the chosen splitting; it is an open question to determine the splitting strategy that minimizes the number of operations. Empirically \changed{it seems} that the unbalanced choice $k=1$ leads to better execution times than the balanced choice $k=\lfloor d/2 \rfloor.$ 

For the choice $k=1$, the upper bound \cref{eqn_recursive_bound} can be expanded as follows.
Recall that if $F_d \le d F_{d-1} + g_d$ with $F_0 = 0$, then
\begin{align}\label{eqn_recursive_bound_solution}
\begin{split}
F_d
\le g_d + d F_{d-1}
&\le g_d + d g_{d-1} + d(d-1) F_{d-2}
\le \cdots \le \sum_{k=0}^{d-1} d^{\underline{k}} g_{d-k},
\end{split}
\end{align}
where $x^{\underline{k}} = x (x-1) \cdots (x-k+1)$ and $x^{\underline{0}}=1$.
Considering \cref{eqn_recursive_bound}, we can bound
\[
\binom{m}{k} \binom{m}{d-k} \binom{d}{k} 
= \frac{m^{\underline{k}}}{k!} \frac{m^{\underline{d-k}}}{(d-k)!} \binom{d}{k} 
\le \frac{m^d}{k! (d-k)!} \binom{d}{k} 
= \frac{m^d}{d!} \binom{d}{k}^2.
\]
Letting $g_d = 4 \frac{m^d}{d!} d^2$, we see that 
\[
d^{\underline{k}} g_{d-k} 
= 4 (d-k)^2 \frac{d^{\underline{k}} m^{d-k}}{(d-k)!} 
= 4 \frac{d^{\underline{k+1}} m^{d-k}}{(d-k-1)!}
\le 4 d^2 m^{d-1} \frac{1}{(d-k-1)!},
\]
having used $d \le m$ in the final step.
Since \cref{eqn_recursive_bound} is of the form \cref{eqn_recursive_bound_solution}, we find
\begin{align*}
 C_{\wedge}^{m,d}
\le \sum_{k=0}^{d-1} d^{\underline{k}} g_{d-k}
\le 4 d^2 m^{d-1} \sum_{k=0}^{d-1} \frac{1}{(d-k-1)!}
\le 4 d^2 m^{d-1} \sum_{k=0}^{\infty} \frac{1}{k!} 
= \BigO{d^2 m^{d-1}}
 \end{align*}
operations as asymptotic time complexity for $m\to\infty$ and fixed $d$.

\bibliographystyle{ACM-Reference-Format} 
\bibliography{dleto-skew}

\end{document}